\documentclass[12pt]{amsart}

\usepackage{amssymb}

\newcommand{\GL}{\mathrm{GL}}
\newcommand{\SL}{\mathrm{SL}}
\newcommand{\Sp}{\mathrm{Sp}}
\renewcommand{\O}{\mathrm{O}}

\newcommand{\G}{\mathbf{G}}
\renewcommand{\H}{\mathbf{H}}

\renewcommand{\P}{\mathbf{P}}
\newcommand{\X}{\mathbf{X}}

\newcommand{\Hom}{\mathrm{Hom}}

\newcommand{\C}{\mathbb{C}}

\newcommand{\Ff}{\mathcal{F}}
\newcommand{\Gg}{\mathcal{G}}

\renewcommand{\span}[1]{\langle #1\rangle}
\newcommand{\Stab}{\mathrm{Stab}}

\newcommand{\Fl}{\mathrm{Fl}}

\newcommand{\length}[1]{|#1|}
\newcommand{\dims}[1]{\Lambda(#1)}

\numberwithin{equation}{section}

\newtheorem{theorem}{Theorem}[section]
\newtheorem{lemma}[theorem]{Lemma}
\newtheorem{corollary}[theorem]{Corollary}
\newtheorem{proposition}[theorem]{Proposition}
\theoremstyle{definition}
\newtheorem{remark}[theorem]{Remark}
\newtheorem{definition}[theorem]{Definition}
\newtheorem{example}[theorem]{Example}
\newtheorem{problem}[theorem]{Problem}

\begin{document}

\title{Multiple flag ind-varieties with finitely many orbits}

\author{Lucas Fresse}
\address{Universit\'e de Lorraine, CNRS, Institut \'Elie Cartan de Lorraine, UMR 7502, Van\-doeuvre-l\`es-Nancy, F-54506 France}
\email{lucas.fresse@univ-lorraine.fr}
\author{Ivan Penkov}
\address{Jacobs University Bremen, Campus Ring 1, 28759 Bremen, Germany}
\email{i.penkov@jacobs-university.de}

\date{\today}

\begin{abstract}
Let $\G$ be one of the ind-groups $\GL(\infty)$, $\O(\infty)$, $\Sp(\infty)$, and let $\P_1,\ldots,\P_\ell$ be an arbitrary set of $\ell$ splitting parabolic subgroups of $\G$. We determine all such sets with the property that $\G$ acts with finitely many orbits on the ind-variety $\X_1\times\cdots\times\X_\ell$ where $\X_i=\G/\P_i$. In the case of a finite-dimensional classical linear algebraic group $G$, the analogous problem has been solved in a sequence of papers of Littelmann, Magyar--Weyman--Zelevinsky and Matsuki. An essential difference from the finite-dimensional case is that already for $\ell=2$, the condition that $\G$ acts on $\X_1\times \X_2$ with finitely many orbits is a rather restrictive condition on the pair $\P_1,\P_2$. We describe this condition explicitly. Using the description we tackle the most interesting case where $\ell=3$, and present the answer in the form of a table. For $\ell\geq 4$ there always are infinitely many $\G$-orbits on $\X_1\times\cdots\times\X_\ell$.
\end{abstract}

\keywords{Classical ind-groups; ind-varieties; generalized flags; multiple flag varieties}
\subjclass[2010]{14L30; 14M15; 22E65; 22F30}

\maketitle

\section*{Introduction}

The following is a fundamental question in the theory of group actions: given a linear reductive algebraic group $G$, on which direct products $X_1\times X_2\times \cdots\times X_\ell$ of compact $G$-homogeneous spaces does $G$ act with finitely many orbits? The problem is non-trivial only for $\ell>2$, since it is a classical fact that $G$ always acts with finitely many orbits on $X_1\times X_2$ (parabolic Schubert decomposition of a partial flag variety). It has turned out that the problem is most interesting for $\ell=3$, as for $\ell\geq 4$ the group $G$ always acts with infinitely many orbits. 

In the special case where one of the factors is a full flag variety, e.g. $X_1=G/B$, the above problem is equivalent to finding whether there are finitely many $B$-orbits on $X_2\times X_3$; this special case is solved in \cite{Littelmann} and \cite{Stembridge}. In this situation, the theory of spherical varieties is an effective tool. 
In particular, the existence of a dense $B$-orbit is sufficient for ensuring that there are finitely many $B$-orbits.
The problem is also related to studying the complexity
of a direct product of two HV-varieties, i.e. closures of $G$-orbits of highest weight vectors in irreducible $G$-modules; this problem is considered in~\cite{Panyushev}.

If no factor $X_i$ is a full flag variety, the problem is considered in the classical cases in \cite{MWZ1,MWZ2} (types A and C, through the theory of quiver representations) and in \cite{Matsuki1,Matsuki2} (types B and D).
For exceptional groups, the general question has been considered in \cite{BDEK}.

We also mention the works \cite{HNOO} and \cite{NO}, where the authors study double flag varieties of the form $G/P\times K/Q$ with a finite number of $K$-orbits
for a symmetric subgroup $K$ of $G$.
The problem of finitely many $G$-orbits on $X_1\times X_2\times X_3$ is recovered if
$K$ is taken to be the diagonal embedding of $G$ into $G\times G$.

In the present paper we address the above general problem in a natural infinite-dimensional setting. We let $\G$ be one of the classical (or finitary) ind-groups $\GL(\infty)$,  $\O(\infty)$, $\Sp(\infty)$ and ask the same question, where each $\X_i$ is now a locally compact $\G$-homogeneous ind-space. The latter are known as ind-varieties of generalized flags and have been studied in particular in \cite{Dimitrov-Penkov} and  \cite{Fresse-Penkov}; see also \cite{DPW} and the references therein. 

For these ind-varieties our question becomes interesting already for $\ell=2$. Indeed, for which direct products $\X_1\times \X_2$ of ind-varieties of generalized flags does $\G$ act with finitely many orbits on $\X_1\times \X_2$? We prove that this is a quite restrictive property of the ind-variety $\X_1\times\X_2$. More precisely, we show that $\G$ acts with finitely many orbits on $\X_1\times\X_2$ only if the stabilizers $\P_1$ and $\P_2$ of two respective (arbitrary) points on $\X_1$ and $\X_2$ have each only finitely many invariant subspaces in the natural representation $V$ of $\G$. In addition, it is required that the invariant subspaces of one of the groups, say $\P_1$, are only of finite dimension or finite codimension. The precise result is Theorem \ref{T1}, where we introduce adequate terminology: we call the parabolic ind-subgroup $\P_1$ large, and the parabolic ind-subgroup $\P_2$ semilarge.

Having settled the case $\ell=2$ in this way, we saw ourselves strongly motivated to solve the problem for any $\ell\geq 3$. The case $\ell\geq 4$ is settled by a general statement, Lemma \ref{L-key}, claiming roughly that in the direct limit case the number of orbits can only increase. Hence for $\ell\geq 4$ there are infinitely many orbits on $\X_1\times\cdots\times\X_\ell$. 
The case $\ell=3$ is the most intriguing. Here we prove that $\X_1\times\X_2\times\X_3$ has finitely many $\G$-orbits, if and only if the same is true for all products
$\X_1\times \X_2$, $\X_2\times \X_3$ and $\X_1\times \X_3$, and in addition
$\X_1\times \X_2\times \X_3$
can be exhausted by triple flag varieties with finitely many orbits over the corresponding finite-dimensional groups. Those triple flag varieties have been classified by Magyar--Weymann--Zelevinsky for $\SL(n)$ and $\Sp(2n)$ \cite{MWZ1,MWZ2}, and by Matsuki for $\O(2n+1)$ and $\O(2n)$ \cite{Matsuki1,Matsuki2}. In this way, we settle the problem completely for the classical ind-groups $\GL(\infty)$, $\O(\infty)$, $\Sp(\infty)$.


\subsection*{Acknowledgement} We are thankful to Roman Avdeev for pointing out the papers \cite{BDEK} and \cite{Matsuki2} to us, and for some constructive remarks. We thank Alan Huckleberry for a general discussion of the topic of this work. L.F. has been supported in part by ANR project GeoLie (ANR-15-CE40-0012). I.P. has been supported in part by DFG Grant PE 980/7-1.

\section{Statement of main results}

\label{S1}

\subsection{Classical ind-groups}
\label{section-1.1}
The base field is $\C$. Let $V$ be a countable-dimensional vector space. Classical ind-groups are realized as subgroups of the group $\GL(V)$ of linear automorphisms of $V$. We consider three situations to which we refer as {\em types}:
\begin{itemize}
\item[(A)] no additional structure on $V$;
\item[(BD)] $V$ is endowed with a nondegenerate symmetric bilinear form $\omega$;
\item[(C)] $V$ is endowed with a symplectic bilinear form $\omega$.
\end{itemize}
The dual space $V^*=\Hom(V,\C)$ is uncountable dimensional. We fix once and for all a countable-dimensional subspace $V_*\subset V^*$ such that the pairing $$V_*\times V\to\C$$ is nondegenerate:
in type (A) we fix any subspace $V_*\subset V^*$ which satisfies these conditions, while
in types (BD) and (C) (type (BCD), for short) we take $V_*:=\{\omega(v,\cdot):v\in V\}$.

Let $\G$ be one of the classical ind-groups
\begin{eqnarray*}
\GL(\infty) & := & \{\parbox[t]{10.3cm}{$g\in\GL(V):$ $g(V_*)=V_*$ and there are finite-codimensional subspaces of $V$ and $V_*$  fixed pointwise by $g$ $\}$,} \\[2mm]
\O(\infty) & := & \{g\in\GL(\infty):\mbox{$g$ preserves $\omega$}\}\quad \mbox{in type (BD),} \\[2mm]
\Sp(\infty) & := & \{g\in\GL(\infty):\mbox{$g$ preserves $\omega$}\}\quad \mbox{in type (C).}
\end{eqnarray*}

To describe $\G$ as an ind-group, we need to take a basis of $V$.
If $E$ is a basis of $V$, we denote by $E^*=\{\phi_e:e\in E\}\subset V^*$ the dual family of linear forms defined by
\[\phi_e(e')=\left\{\begin{array}{ll} 0 & \mbox{if $e'\in E\setminus\{e\}$,} \\ 1 & \mbox{if $e'=e$.} \end{array}\right. \]
We call $E$ {\em admissible} if, according to type, the following is satisfied:
\begin{itemize}
\item[(A)] the dual family $E^*$ spans the subspace $V_*$;
\item[(BCD)] $E$ is endowed with an involution $i_E:E\to E$, with at most one fixed point, such that $\omega(e,e')\not=0$ if and only if $e'=i_E(e)$.
\end{itemize}
If $E$ is admissible in the sense of type (BCD), then it is a fortiori admissible in the sense of type (A). Note that in type (C), the involution $i_E$ cannot have a fixed point.


We claim now that, for any admissible basis $E$, in type (A) the group $\G=\GL(\infty)$ coincides with the group 
\[
\GL(E):=\{g\in\GL(V):g(e)=e\mbox{ for almost all $e\in E$}\},
\]
where ``almost all'' means ``all but finitely many''.
Indeed, clearly, $\GL(E)$ is a subgroup of $\GL(\infty)$. For the opposite inclusion, consider $g\in\GL(\infty)$.
Since 
$g$ fixes pointwise a finite-codimensional subspace of $V$, there exists a cofinite subset $E'\subset E$ such that 
$g(e')-e'=\sum_{e\in E\setminus E'}x_{e,e'}e\in\span{E\setminus E'}$ 
for all $e'\in E'$. On the other hand,
the fact that $g^{-1}(\phi_e)\in V_*$ implies $x_{e,e'}=0$ for 
all $e\in E\setminus E'$ and almost all $e'$, which shows that $g\in \GL(E)$.

In type (BCD), we have
\[
\G=\{g\in\GL(V):g(e)=e\mbox{ for almost all $e\in E$, and $g$ preserves $\omega$}\}.
\]
If we take any filtration $E=\bigcup_{n\geq 1}E_n$ by finite subsets, such that $E_n$ is $i_E$-stable in type (BCD), 
we get an exhaustion of $\G$,
\[
\G=\bigcup_{n\geq 1}\G(E_n)\quad\mbox{for $\G(E_n):=\G\cap\GL(\langle E_n\rangle)$}.
\]
Here the notation $\langle\cdot\rangle$ stands for the linear span,
and $\GL(\langle E_n\rangle)$ is viewed as a subgroup of $\GL(V)$ in the natural way. 
The subgroups $\G(E_n)$ are finite-dimensional algebraic groups isomorphic to $\GL_m(\C)$, $\O_m(\C)$, or $\Sp_m(\C)$, depending on whether $\G$ is $\GL(\infty)$, $\O(\infty)$, or $\Sp(\infty)$. This exhaustion provides $\G$ with a structure of ind-group, which is independent of the chosen admissible basis. 

\subsection{Splitting Cartan and parabolic subgroups} 
\label{section-1.2}
We call $\H\subset \G$ a {\em splitting Cartan subgroup} if it is the subgroup $\H(E)$ of elements which are diagonal in some admissible basis $E$.

We call $\P\subset\G$ a {\em splitting parabolic subgroup} if it contains a splitting Cartan subgroup $\H=\H(E)$, for some admissible basis $E$, and $\P(E_n):=\P\cap \G(E_n)$ is a parabolic subgroup of $\G(E_n)$ for all $n\geq 1$.

Splitting parabolic subgroups can be fully classified in terms of so-called generalized flags; see Section \ref{section-2.2}. A splitting Borel subgroup is a splitting parabolic subgroup which is minimal (equivalently this is the stabilizer of a generalized flag which is maximal). We consider two types of splitting parabolic subgroups which are not Borel subgroups: 

\begin{definition}
\label{D1}
We say that a splitting parabolic subgroup $\P$ is {\em semilarge} if it has only finitely many invariant subspaces in $V$. This is equivalent to the requirement that $\P$ be the stabilizer in $\G$ of a finite sequence of subspaces
\[
\{F_0=0\subsetneq F_1\subsetneq F_2\subsetneq\ldots\subsetneq F_{m-1}\subsetneq F_m=V\}
\]
(such that $F_k^\perp=F_{m-k}$ in type (BCD)).

We say that $\P$ is {\em large} if, moreover, each subspace $F_k$ is either finite dimensional or finite codimensional.
\end{definition}

Given a splitting parabolic subgroup $\P\subset\G$, the quotient set $\G/\P$ has a structure of ind-variety, which is given by the exhaustion
\[
\G/\P=\bigcup_{n\geq 1}\G(E_n)/\P(E_n).
\]
Each quotient $\G(E_n)/\P(E_n)$ is a flag variety for the group $\G(E_n)$, hence a projective variety. Thus $\G/\P$ is locally projective, but in general it is not projective, i.e. does not admit an embedding as a closed ind-subvariety in the infinite-dimensional projective space $\mathbb{P}^\infty(\C)$. However, if $\P$ is large or semilarge, such an embedding does exist (see \cite[Proposition 7.2]{Dimitrov-Penkov}). It is worth to note that, for any splitting parabolic subgroup $\P$, the ind-variety $\G/\P$ can be realized as an ind-variety of generalized flags; see Section \ref{section-2.3}.

Contrary to the finite-dimensional situation, any two splitting Cartan subgroups of $\G$ do not have to be conjugate;
see Example \ref{E2}. In this paper, a source of difficulty is that we are considering splitting parabolic subgroups which do not a priori have a splitting Cartan subgroup in common, even up to conjugacy. The following characterization of large splitting parabolic subgroups will be useful in this respect.

\begin{proposition}[see Proposition \ref{P1bis}]
\label{P1}
Let $\P\subset\G$ be a splitting parabolic subgroup. The following conditions are equivalent.
\begin{itemize}
\item[\rm (i)] $\P$ is large;
\item[\rm (ii)] For every splitting Cartan subgroup $\H\subset\P$, there is $g\in \G$ such that $g\H g^{-1}\subset\P$.
\end{itemize}
\end{proposition}

The proof is given in Section \ref{section-3}.

\subsection{Main results} 
We consider a product of ind-varieties of the form
\begin{equation}
\label{X}
\X=\G/\P_1\times\cdots\times\G/\P_\ell
\end{equation}
where $\P_1,\ldots,\P_\ell\subsetneq\G$ are splitting parabolic subgroups of $\G$. 
The ind-variety $\X$ is equipped with the diagonal action of $\G$. Our purpose is to solve the following problem:

\begin{problem}
\label{Problem1}
Characterize all $\ell$-tuples $(\P_1,\ldots,\P_\ell)$ such that $\X$ has a finite number of $\G$-orbits.
\end{problem}

Of course if $\ell=1$, then $\X$ has only one $\G$-orbit. 
If $\ell=2$, the number of orbits in $\X$ is infinite in general, and our first main result claims the following.

\begin{theorem}
\label{T1}
If $\ell=2$, then $\X$ (of (\ref{X})) has a finite number of $\G$-orbits if and only if
one of the subgroups $\P_1,\P_2$ is large and the other one is semilarge.
\end{theorem}

\begin{corollary}
Let $\P$ be a splitting parabolic subgroup of $\G$.\ Then the ind-variety $\G/\P$ has a finite number of $\P$-orbits if and only if $\P$ is large.
\end{corollary}

Next we consider the case $\ell=3$, i.e.
\[\X=\G/\P_1\times \G/\P_2\times \G/\P_3.\]
By Theorem \ref{T1},
if $\X$ has a finite number of $\G$-orbits then all three splitting parabolic subgroups $\P_1,\P_2,\P_3$ are semilarge and at least two of them are large.
Moreover, it follows from Proposition \ref{P1} that, up to replacing the parabolic subgroups by conjugates, there is no loss of generality in assuming that $\P_1,\P_2$, and $\P_3$ contain the same splitting Cartan subgroup $\H=\H(E)$ for some admissible basis $E$.
 This assumption guarantees that the construction of Section \ref{section-1.2} can be done simultaneously for each factor $\G/\P_i$ ($i\in\{1,2,3\}$). Hence,
 by considering a filtration $E=\bigcup_n E_n$ as in Section \ref{section-1.2}, we obtain an exhaustion 
\begin{equation}
\label{X-exhaustion}
\X=\bigcup_{n\geq 1}\X(E_n)
\end{equation}
where
$\X(E_n):=\prod_{i=1}^3 \G(E_n)/\P_i(E_n)$ is a triple flag variety for the group $\G(E_n)$. See Section \ref{exhaustion} for more details. 

Our main result regarding the case $\ell=3$ can be stated as follows.

\begin{theorem} 
\label{T2}
If $\ell=3$, then $\X$ can have a finite number of $\G$-orbits only if 
the splitting parabolic subgroups $\P_1,\P_2,\P_3$ are semilarge and at least two of them are large.
Moreover, in this situation, $\X$ has a finite number of $\G$-orbits if and only if, for all $n$, the finite-dimensional triple flag variety $\X(E_n)$ has a finite number of $\G(E_n)$-orbits.

Finally, $\X$ has a finite number of $\G$-orbits if and only if
the triple $(\P_1,\P_2,\P_3)$ appears, up to permutation, in Table~\ref{table1}.
\end{theorem}

In Table \ref{table1}, we use the following notation for a semilarge parabolic subgroup $\P$
obtained as the stabilizer of a finite chain of subspaces $\{F_0=0\subsetneq F_1\subsetneq\ldots\subsetneq F_m=V\}$ as in Definition \ref{D1}. 
We set $\length{\P}=m$
and denote by $\dims{\P}$ the list of values $\dim F_k/F_{k-1}$ ($k=1,\ldots,m$) written in nonincreasing order; in case of $j$ repetitions of the same value $a$, we write $a^j$. Note that this list always starts with $\infty$, and $\P$ is large if and only if there is a unique occurrence of $\infty$ in the list.

\begin{table}[h]

$\GL(\infty)$ case:

{\scriptsize \begin{tabular}{|c|c|c|c|}
\hline
$\P_1$ & $\P_2$ & $\P_3$ & additional condition \\
\hline\hline
$\length{\P_1}=2$ & $\length{\P_2}=2$ & semilarge & two of $\P_1,\P_2,\P_3$ are large \\
\hline 
$\length{\P_1}=2$ & $\length{\P_2}=3$ & $3\leq \length{\P_3}\leq 5$ & two of $\P_1,\P_2,\P_3$ are large \\
\hline
$\dims{\P_1}=(\infty,2)$ & $\length{\P_2}=3$ & semilarge & $\P_2$ or $\P_3$ is large \\
\hline
$\length{\P_1}=2$ & $\length{\P_2}=3$, $1\in\dims{\P_2}$ & semilarge & two of $\P_1,\P_2,\P_3$ are large \\
\hline
$\dims{\P_1}=(\infty,1)$ & large & semilarge & no additional condition \\ \hline

\end{tabular}}

\bigskip
$\Sp(\infty)$ case:

{\scriptsize \begin{tabular}{|c|c|c|c|}
\hline
$\P_1$ & $\P_2$ & $\P_3$ \\
\hline\hline
$\length{\P_1}=2$ & $\length{\P_2}=3$, large & $\length{\P_3}\in\{3,5\}$, large \\
\hline 
$\length{\P_1}=2$ & $\dims{\P_2}=(\infty,1^2)$ & large  \\
\hline
$\dims{\P_1}=(\infty,1^2)$ & $\length{\P_2}=3$ & large  \\ \hline
$\dims{\P_1}=(\infty,1^2)$ & $\length{\P_2}=3$, large & semilarge  \\ \hline
\end{tabular}}

\bigskip
$\O(\infty)$ case:

{\scriptsize \begin{tabular}{|c|c|c|c|}
\hline
$\P_1$ & $\P_2$ & $\P_3$ \\
\hline\hline
$\length{\P_1}=2$ & $\dims{\P_2}=(\infty,b^2)$, $b\leq3$ & large \\
\hline 
$\length{\P_1}=2$ & $\dims{\P_2}=(\infty,1^4)$ & large  \\
\hline
$\length{\P_1}=2$ & $\length{\P_2}=3$, large & $\length{\P_3}\in\{3,5\}$, large  \\ \hline
$\length{\P_1}=2$ & $\length{\P_2}=3$, large & $\dims{\P_3}=(\infty,c^2,1^4)$, $c<\infty$  \\ \hline
$\length{\P_1}=2$ & $\length{\P_2}=3$, large & $\dims{\P_3}=(\infty,1^8)$ \\ \hline
$\dims{\P_1}=(\infty,b^2)$, $b<\infty$ & $\dims{\P_2}=(\infty,\infty,1)$ & $\length{\P_3}\in\{3,5\}$, large \\ \hline
$\dims{\P_1}=(\infty,1^2)$ & $\length{\P_2}=3$, large & semilarge  \\ \hline
$\dims{\P_1}=(\infty,1^2)$ & $\length{\P_2}\in\{3,4\}$ & large \\ \hline
\end{tabular}}

\bigskip

\caption{Classification of triples $(\P_1,\P_2,\P_3)$, up to permutation, such that $\G/\P_1\times\G/\P_2\times\G/\P_3$ is of finite type.}
\label{table1}
\end{table}

\bigskip

Finally, it is not surprising that,  like in the case of finite-dimensional multiple flag varieties, the following holds.

\begin{theorem}
\label{T3}
If $\ell\geq 4$, then $\X$ has an infinite number of $\G$-orbits.
\end{theorem}

The rest of the paper is structured as follows.
In Section \ref{S2} we summarise some existing results on the classical ind-groups and their homogeneous ind-varieties.  
In Section \ref{section-3} we show the characterization of large parabolic subgroups stated in Proposition \ref{P1}.
In Section \ref{section-4} we explain the construction of the exhaustion of (\ref{X-exhaustion})
 in more detail, and prove Lemma \ref{L-key}
which claims that whenever the multiple ind-variety $\X$ of (\ref{X}) has an exhaustion as in (\ref{X-exhaustion}),
we get an embedding of orbit sets
\[
\X(E_n)/\G(E_n)\hookrightarrow \X/\G.
\]
This lemma plays a key role in the proof of our main results.
Theorem \ref{T1} is proved in Sections \ref{section-5}--\ref{section-6}. 
Theorem \ref{T2} is proved in Section \ref{section-7}, except for the verification of Table \ref{table1} which is done in Appendix \ref{appendix}.
Finally the proof of Theorem \ref{T3} appears in the very short Section~\ref{section-8}.

\section{Preliminaries on admissible bases and generalized flags}

\label{S2}


\subsection{Splitting Cartan subgroups and admissible bases}

\label{section-2.1}

We refer to Sections \ref{section-1.1} and \ref{section-1.2} for the definitions of admissible basis and splitting Cartan subgroup.
Note that the group $\G$ acts on the set of admissible bases. Moreover, if $E$ is an admissible basis of $V$ then \[
\H(g(E))=g\H(E)g^{-1},
\]
hence $\G$ acts by conjugation on the set of splitting Cartan subgroups.

\begin{lemma}
\label{L2.1}
Let $E,E'$ be two admissible bases of $V$ which differ by finitely many vectors, that is, 
\[E=E_0\sqcup I\quad\mbox{and}\quad E'=E_0\sqcup I'\]
where $I,I'$ are finite sets. Then the splitting Cartan subgroups $\H(E)$ and $\H(E')$ are conjugate. 

Conversely, if $\H$ and $\H'$ are two conjugate splitting Cartan subgroups,
then there are admissible bases $E$ and $E'$ which differ by finitely many vectors such that $\H=\H(E)$ and $\H'=\H(E')$.
\end{lemma}

\begin{proof} 
First, we note that $I$ and $I'$ have the same cardinality, equal to the codimension of $\span{E_0}$ in $V$.
In type (A) we take $g\in\GL(V)$ such that $g(I)=I'$ and $g(e)=e$ for all $e\in E_0$, hence $g(E)=E'$. This element $g$ actually belongs to $\GL(\infty)=\G$, and we get $\H(E')=\H(g(E))=g\H(E)g^{-1}$.

In type (BCD), 
up to considering larger $I$ and $I'$ if necessary, we may assume that $I$ and $I'$ are stable by the involutions $i_E$ and $i_{E'}$, respectively. Since $I$ and $I'$ have the same cardinality and the involutions $i_E$ and $i_{E'}$ have at most one fixed point, we can write either
\[
I=\{e_1,\ldots,e_k,e_1^*,\ldots,e_k^*\}\quad\mbox{and}\quad I'=\{e'_1,\ldots,e'_k,e_1'^*,\ldots,e_k'^*\},
\]
or
\[
I=\{e_0,e_1,\ldots,e_k,e_1^*,\ldots,e_k^*\}\quad\mbox{and}\quad I'=\{e'_0,e'_1,\ldots,e'_k,e_1'^*,\ldots,e_k'^*\},
\]
with $e_i^*=i_E(e_i)$, $e_i'^*=i_{E'}(e_i')$ for $1\leq i\leq k$ and $e_0=i_E(e_0)$, $e'_0=i_{E'}(e'_0)$.
Up to replacing the vectors of $I$ and $I'$ by scalar multiples (which does not change the splitting Cartan subgroups) we may assume that
\[\omega(e_i,e_i^*)=\omega(e'_i,e_i'^*)=1\ \mbox{for $1\leq i\leq k$},\quad \omega(e_0,e_0)=\omega(e'_0,e'_0)=1.\]
Then, by letting $g(e)=e$ for $e\in E_0$ and $g(e_i)=e'_i$, $g(e_j^*)=e_j'^*$ for all $i,j$, we get an element $g\in\G$ such that $\H(E')=\H(g(E))=g\H(E)g^{-1}$.

The converse statement follows by observing that, if $E$ is an admissible basis such that $\H=\H(E)$, and $g\in\G$ satisfies $\H'=g\H g^{-1}$, then $\H'=\H(g(E))$, where $g(E)$ is an admissible basis which differs from $E$ by finitely many vectors.
\end{proof}

\begin{remark}
\label{R2.2}
{\rm (a)} In type (A), if $E$ is an admissible basis of $V$, then any basis $E'$ which differs from $E$ by finitely many vectors is admissible.
Indeed, in this case we have an element $g\in\G$ such that $E'=g(E)$.

{\rm (b)} In type (BCD), part (a) of the remark does not hold, but we point out the following construction of admissible basis.
Let an orthogonal decomposition $V=V_1\oplus V_2$ be given.
The restriction $\omega_i$ of the form $\omega$ to each subspace $V_i$ ($i\in\{1,2\}$) is nondegenerate. For $i\in\{1,2\}$, let $E_i$ be an admissible basis of $V_i$, that is,
a basis endowed with an involution $i_{E_i}:E_i\to E_i$ with at most one fixed point such that $\omega(e,e')\not=0$ if only if $e'=i_{E_i}(e)$.
Moreover, assume that $i_{E_1}$ or $i_{E_2}$ has no fixed point. Then $E_1\cup E_2$ is an admissible basis of $V$ for the involution $i_{E_1}\cup i_{E_2}$.

{\rm (c)} In type (BCD) any admissible basis can be written as
\[E=\{e_n,e_n^*\}_{n\geq 1}
\quad\mbox{or}\quad
E=\{e_0\}\cup\{e_n,e_n^*\}_{n\geq 1}\]
where $e_n^*=i_E(e_n)$ for all $n\geq 1$ and $e_0=e_0^*$ is the fixed point of $i_E$ (if it exists).
By replacing the vectors by scalar multiples (which does not change the splitting Cartan subgroup $\H(E)$), we can transform $E$ into a basis with
\[\omega(e_m,e_n^*)=\delta_{m,n}\quad\mbox{for all $m,n$}.\]
In \cite{Dimitrov-Penkov}, a basis which satisfies this property is called $\omega$-isotropic.\ If $\omega$ is symplectic, an $\omega$-isotropic basis is said to be a basis of type (C). If $\omega$ is symmetric, an $\omega$-isotropic basis is called of type (B) or (D) depending on whether $i_E$ has a fixed point or no fixed point. 

In type (BD), bases of both types (B) and (D) do exist in $V$ and their corresponding splitting Cartan subgroups cannot be conjugate.
\end{remark}

The following example shows that, in any type, there are splitting Cartan subgroups which are not conjugate. In fact, using the construction made in this example, it is easy to show that there are infinitely many conjugacy classes of splitting Cartan subgroups.

\begin{example}
\label{E2}
Let $\H=\H(E)\subset \G$ be a splitting Cartan subgroup, associated to an admissible basis. Let $I=\{e_n,e'_n\}_{n\geq 1}$ be a double infinite sequence of (pairwise distinct) vectors of $E$, moreover in type (BCD) we assume that these vectors are pairwise orthogonal, that is,
\begin{equation}
\label{Iisotropic}
\mbox{$\{e_n,e'_n\}_{n\geq 1}$ spans an isotropic subspace of $V$.}
\end{equation}
We construct a splitting Cartan subgroup $\H'\subset\G$ such that
\begin{equation}
\label{PH'}
\forall n\geq 1,\ \exists h\in\H'\ \mbox{such that}\ h(e_n)=e'_n\ \mbox{and}\ h(e'_n)=e_n.
\end{equation}
The subgroup $\H'$ cannot be conjugate to $\H$: for every $g\in\G$, we have $g(e_n)=e_n$ and $g(e'_n)=e'_n$ whenever $n\geq 1$ is large enough,
and (\ref{PH'}) yields $h\in\H'$ with
\[ghg^{-1}(e_n)=gh(e_n)=g(e'_n)=e'_n\notin\span{e_n}.\]
Hence, $g\H'g^{-1}\not\subset\H$.

For constructing $\H'$, we construct an admissible basis $\tilde{E}$ of $V$ which contains the vectors $\tilde{e}_n:=e_n+e'_n$ and $\tilde{e}'_n:=e_n-e'_n$, for all $n\geq 1$, and then we define $\H'$ as the subgroup $\H(\tilde{E})\subset\G$ of all elements which are diagonal in the basis $\tilde{E}$. This subgroup fulfills (\ref{PH'}),
since for all $n\geq 1$ we can find $h\in\H'$
such that $h(\tilde{e}_n)=\tilde{e}_n$ and $h(\tilde{e}'_n)=-\tilde{e}'_n$.

The construction of $\tilde{E}$ is done as follows. In type (A), we take
\[
\tilde{E}:=(E\setminus I)\cup \{\tilde{e}_n,\tilde{e}'_n\}_{n\geq 1}.
\]
The dual family $\tilde{E}^*=\{\tilde\phi_e:e\in \tilde{E}\}$ consists of the linear functions
\[
\tilde\phi_e:=\left\{\begin{array}{ll}
\phi_e & \mbox{if $e\in E\setminus I$,} \\
\frac{1}{2}(\phi_{e_n}+\phi_{e'_n}) & \mbox{if $e=\tilde{e}_n$,} \\
\frac{1}{2}(\phi_{e_n}-\phi_{e'_n}) & \mbox{if $e=\tilde{e}'_n$,}
\end{array}\right.
\]
where $E^*=\{\phi_e:e\in E\}$ is the dual family of $E$. Hence $\langle \tilde{E}^*\rangle=\langle E^*\rangle$, and this shows that $\tilde{E}$ is admissible.

In type (BCD), we note first that condition (\ref{Iisotropic}) implies that the double sequence $I$ and its image by the involution $i_E:E\to E$ are pairwise disjoint. Then we set
\[
\tilde{E}:=(E\setminus (I\cup i_E(I))\cup\{\tilde{e}_n,\tilde{e}'_n\}_{n\geq 1}\cup\{\tilde{f}_n,\tilde{f}'_n\}_{n\geq 1}
\]
where
\begin{eqnarray*}
\tilde{f}_n & := & \frac{i_E(e_n)}{\omega(e_n,i_E(e_n))}+\frac{i_E(e'_n)}{\omega(e'_n,i_E(e'_n))},\\
\tilde{f}'_n & := & \frac{i_E(e_n)}{\omega(e_n,i_E(e_n))}-\frac{i_E(e'_n)}{\omega(e'_n,i_E(e'_n))}. \end{eqnarray*}
It is easy to check that the basis $\tilde{E}$ so-obtained is admissible, with involution $i_{\tilde{E}}:\tilde{E}\to\tilde{E}$ given by
$i_{\tilde{E}}(\tilde{e}_n)=\tilde{f}_n$,
$i_{\tilde{E}}(\tilde{e}'_n)=\tilde{f}'_n$ for all $n\geq 1$
and $i_{\tilde{E}}\equiv i_E$ on $E\setminus(I\cup i_E(I))$.
\end{example}

For later use, we also point out that an element of $\G$ (contrary to a general element of $\GL(V)$) cannot map a subspace $A\subset V$ onto a larger subspace.

\begin{lemma}
\label{L-new}
Given subspaces $A\subsetneq B\subset V$, there is no $g\in\G$ such that $g(A)=B$.
\end{lemma}

\begin{proof}
Arguing by contradiction, assume that there exists $g\in\G$ with $g(A)=B$.
Let $E$ be an admissible basis of $V$. There is a finite subset $E_0\subset E$, such that $g(e)=e$ for all $e\in E\setminus E_0$ and $g$ stabilizes the finite-dimensional subspace $\span{E_0}$. Up to choosing the finite subset $E_0$ larger if necessary, we may assume that $A\cap\span{E_0}\subsetneq B\cap\span{E_0}$, while the equality $g(A\cap\span{E_0})=B\cap\span{E_0}$ holds. This is a contradiction.
\end{proof}

\subsection{Splitting parabolic subgroups and generalized flags}

\label{section-2.2}

The notion of splitting parabolic subgroup can be described in a more handy way by using a model from linear algebra, based on the following definition.

\begin{definition}
\label{D4}
{\rm (a)} A {\em generalized flag} in $V$ is a collection $\mathcal{F}$ of linear subspaces of $V$ which satisfies the following conditions:
\begin{itemize}
\item The inclusion relation $\subset$ is a total order on $\Ff$;
moreover every subspace $F\in\Ff$ has an immediate predecessor or an immediate successor in $\mathcal{F}$ with respect to $\subset$.
\item For every nonzero vector $v\in V$, there is a pair of consecutive subspaces $F',F''\in\mathcal{F}$ such that $v\in F''\setminus F'$.
\item In type (BCD) we require a generalized flag to be {\em isotropic} in the following sense: for every $F\in\Ff$ we have $F^\perp\in\Ff$, and the map $i_\Ff:F\mapsto F^\perp$ is an involution of $\Ff$.
\end{itemize}
Note that the group $\G$ acts on the respective set of generalized flags in a natural way.



{\rm (b)} Let $E$ be an admissible basis. A generalized flag $\mathcal{F}$ in $V$ is said to be {\em $E$-compatible} if it is $\H(E)$-fixed.\ This is equivalent to requiring that each subspace $F\in\Ff$ is spanned by a subset of $E$.

{\rm (c)}
We say that $\Ff$ is {\em weakly $E$-compatible} if it is compatible with a basis $E'$ of $V$ such that $E'\setminus E$ and $E\setminus E'$ are finite (that is, $E$ and $E'$ differ by finitely many vectors).
\end{definition}

\begin{example}
\label{E1}
Generalized flags can take various forms:

{\rm (a)} A generalized flag can be a finite sequence of subspaces $\Ff=\{F_0=0\subset F_1\subset\ldots\subset F_m=V\}$.


{\rm (b)} Let $F\subset V$ be a subspace (taken isotropic in type (BCD)). In type (A), we set $\Ff_F:=\{0\subset F\subset V\}$; this is the minimal generalized flag which contains the subspace $F$. In type (BCD), we set $\Ff_F:=\{0\subset F\subset F^\perp\subset V\}$; this is the minimal isotropic generalized flag which contains $F$. In each case we call $\Ff_F$ the {\em generalized flag associated to $F$}.

{\rm (c)}
$\Ff=\{F_0=0\subset F_1\subset F_2\subset\cdots\subset F_n\subset\ldots\}$ is a generalized flag 
(in type (A))
provided that $\bigcup_{n}F_n=V$.

{\rm (d)}
$\Ff=\{\ldots \subset F_{-n}\subset \cdots\subset F_{-1}\subset F_0\subset F_1\subset\cdots\subset F_n\subset\cdots \}$ is a generalized flag (in type (A)) if $\bigcap_{n}F_n=0$ and $\bigcup_{n}F_n=V$.
It is a generalized flag in type (BCD) if and only if there is some $n_0$ such that $F_n^\perp=F_{n_0-n}$ for all $n$.

{\rm (e)}
Let $E=\{e_x\}_{x\in\mathbb{Q}}$ be a basis of $V$ indexed by the rational numbers. For $x\in\mathbb{Q}$ let $F'_x:=\span{e_y:y<x}$ and $F''_x:=\span{e_y:y\leq x}$. Then $\Ff:=\{F'_x,F''_x\}_{x\in\mathbb{Q}}$ is an $E$-compatible generalized flag (in type (A)) such that each subspace lacks either an immediate predecessor or an immediate successor.

\end{example}

\begin{proposition}[{\cite{Dimitrov-Penkov}}]
\label{P0}
Let $E$ be an admissible basis and $\Ff$ be an $E$-compatible generalized flag in $V$.\ Then the subgroup
\[\P_\Ff=\mathrm{Stab}_\G(\Ff):=\{g\in \G:g\Ff=\Ff\}\]
is a splitting parabolic subgroup of $\G$ that contains the Cartan subgroup $\H(E)$. Moreover, every parabolic subgroup of $\G$ that contains $\H(E)$ is obtained in this way.

In addition, $\P_\Ff$ is a splitting Borel subgroup if and only if the generalized flag $\Ff$ is maximal, in the sense that $\dim F''/F'=1$ for each pair of consecutive subspaces $\{F',F''\}$ in $\Ff$.
\end{proposition}

\begin{remark}
Contrary to the finite-dimensional situation, two splitting Borel subgroups of $\G$ are not necessarily $\G$-conjugate, even if they contain the same splitting Cartan subgroup. This observation follows from Proposition \ref{P0} and from the fact that two $E$-compatible maximal generalized flags $\Ff,\Gg$ do not belong to the same $\G$-orbit in general. For instance, $\Ff$ and $\Gg$ certainly belong to different $\G$-orbits if they are not isomorphic as totally ordered sets. However, even if they are isomorphic as totally ordered sets, 
the generalized flags $\Ff$ and $\Gg$ do not have to be $\G$-conjugate.
For instance if $\Ff=\{\span{e_1,\ldots,e_n}\}_{n\geq 0}$ and
$\Gg=\{\span{e_1,\ldots,e_{2n}},\span{e_1,\ldots,e_{2n},e_{2n+2}}\}_{n\geq 0}$
(where $E=\{e_k\}_{k\geq 1}$), then $\Stab_\G(\Ff)$ and $\Stab_\G(\Gg)$ are two splitting Borel subgroups of $\G=\GL(\infty)$ which are not conjugate.
Indeed every $g\in \G$ satisfies $g(\span{e_1,\ldots,e_n})=\span{e_1,\ldots,e_n}$ for large $n$, hence $\Gg\notin\G\cdot\Ff$.



\end{remark}

\subsection{Ind-varieties of generalized flags}
\label{section-2.3}

Fix an admissible basis $E$ of $V$ and a generalized flag $\Ff=\{F_\alpha\}$ in $V$, compatible with $E$. 
Let $\P=\P_\Ff\subset \G$ be the splitting parabolic subgroup obtained as the stabilizer of $\Ff$, like in Proposition \ref{P0}. In this section, we describe the homogeneous space $\G/\P=\G\cdot\Ff$ as an ind-variety of generalized flags.
See \cite{Dimitrov-Penkov} for more details.

\begin{definition}
A generalized flag $\Gg$ is said to be {\em $E$-commensurable} with $\Ff$ if $\Gg=\{G_\alpha\}$ is parameterized by the same ordered set as $\Ff$, and in addition satisfies the following conditions:
\begin{itemize}
\item $\Gg$ is weakly $E$-compatible;
\item there exists a finite-dimensional subspace $U\subset V$, such that $F_\alpha+U=G_\alpha+U$ and $\dim F_\alpha\cap U=\dim G_\alpha\cap U$ for any $\alpha$.
\end{itemize}
\end{definition}

Let $\Fl(\Ff,E,V)$ denote the set of all generalized 
flags $\Gg$ that are $E$-commensurable with $\Ff$.
Let $\Fl_\omega(\Ff,E,V)$ denote the subset of all such generalized flags which are isotropic.
Then the homogeneous space $\G/\P_\Ff=\G\cdot\Ff$ coincides with the set $\Fl(\Ff,E,V)$ in type (A),
respectively with $\Fl_\omega(\Ff,E,V)$ in type (BCD). (Note that the notion of commensurability is the same whatever type is considered.)

In Section \ref{section-1.2} we notice that the quotient $\G/\P_\Ff$ has the structure of an ind-variety,
obtained by considering a filtration 
\[E=\bigcup_{n\geq 1} E_n\] 
by finite subsets; in type (BCD) the basis is endowed with the involution $i_E:E\to E$ (with at most one fixed point) and we require the subsets $E_n$ to be $i_E$-stable, so that the restriction of the form $\omega$ to each subspace $\span{E_n}$ is nondegenerate.

The ind-structure on $\Fl(\Ff,E,V)$ and $\Fl_\omega(\Ff,E,V)$ is given via the identification with a direct limit
\begin{equation}
\label{Fl-Flo}
\Fl(\Ff,E,V)=\lim_\to \Fl(\Ff,E_n),\quad \Fl_\omega(\Ff,E,V)=\lim_\to \Fl_\omega(\Ff,E_n),
\end{equation}
where $\Fl(\Ff,E_n)$ and $\Fl_\omega(\Ff,E_n)$ are 
varieties of partial flags of the space $V_n:=\span{E_n}$ defined in the following way.
The generalized flag $\Ff$ gives rise to a flag in the finite-dimensional subspace $\langle E_n\rangle$, namely
let $\Ff(n)$ be the collection of subspaces
\begin{equation}
\label{F(n)}
\Ff(n):=\{F\cap \span{E_n}\}_{F\in\Ff}.
\end{equation}
Let $d(\Ff,n)$ denote the corresponding dimension vector
\[d(\Ff,n):=\{d(\Ff,n)_F\}_{F\in\Ff}\quad\mbox{where}\quad d(\Ff,n)_F:=\dim F\cap \span{E_n}.\]
The (finite-dimensional) algebraic variety 
\[
\X_\Ff(n):=\G(E_n)\cdot\Ff(n)\cong \G(E_n)/\P_\Ff(E_n)
\]
can be viewed as the set of collections of nested subspaces of $\span{E_n}$
\[\X_\Ff(n)=
\Fl(\Ff,E_n) := 
\big\{\{M_F\}_{F\in \Ff}:\forall F\in\Ff,\  \dim M_F=d(\Ff,n)_F\big\}
\]
in type (A), respectively
\[\X_\Ff(n)=
\Fl_\omega(\Ff,E_n) := \big\{\{M_F\}_{F\in \Ff}\in\Fl(\Ff,E_n):\forall F\in\Ff,\  (M_F)^\perp=M_{F^\perp}\big\}
\]
in type (BCD).

For each $n\geq 1$, we have the embedding
\[
\phi_\Ff(n):\X_\Ff(n)\hookrightarrow\X_\Ff(n+1),\ \{M_F\}_{F\in \Ff}\mapsto\{N_F\}_{F\in\Ff}
\]
given by
\[
N_F=M_F\oplus(F\cap \span{E_{n+1}\setminus E_n})\quad\mbox{for all $F\in\Ff$.}
\]
Finally, the ind-variety $\G/\P_\Ff=\G\cdot\Ff=\Fl(\Ff,E,V)$, respectively $\Fl_\omega(\Ff,E,V)$, is obtained as the limit of the inductive system
\[
\X_\Ff(1)\hookrightarrow\X_\Ff(2)\hookrightarrow\cdots\hookrightarrow\X_\Ff(n)\hookrightarrow\X_\Ff(n+1)\hookrightarrow\cdots.
\]
This yields (\ref{Fl-Flo}).


\begin{remark}
\label{R-section2.3}
In (\ref{F(n)}),
$\Ff(n)$ is an a priori infinite collection of subspaces of $\span{E_n}$ with repetitions.
If we avoid repetitions, we can also write $\Ff(n)$ as an increasing sequence of subspaces
\[
\{F_0=0\subsetneq F_1\subsetneq \ldots\subsetneq F_s=\span{E_n}\}
\]
and the dimension vector $d(\Ff,n)$ as an increasing sequence
\[
\hat{d}(\Ff,n):=\{d_0=0<d_1<\ldots<d_s\}=\{\dim F_k\}_{k=0}^s.
\]
Then $\Fl(\Ff,E_n)$ can be identified with the variety of partial flags in $V_n=\span{E_n}$,
\[
\Fl(\hat{d}(\Ff,n),V_n)=\big\{
\{M_0\subset M_1\subset \ldots\subset M_s=V_n\}:\forall k,\ \dim M_k=d_k
\big\},
\]
while in type (BCD), $\Fl_\omega(\Ff,E_n)$ is identified with the variety of 
isotropic partial flags 
\[
\Fl_\omega(\hat{d}(\Ff,n),V_n)=\big\{
\{M_k\}_{k=0}^s\in \Fl(\hat{d}(\Ff,n),V_n):\forall k,\ (M_k)^\perp=M_{s-k}
\big\}.
\]

In type (A),
the embedding $\phi_\Ff(n)$ corresponds to an embedding of partial flag varieties
\[
i_n:\Fl(\hat{d}(\Ff,n),V_n)\hookrightarrow\Fl(\hat{d}(\Ff,n+1),V_{n+1})
\]
obtained in the following way.
Assume that $E_{n+1}=E_n\cup\{e\}$ for simplicity (in the general case, $i_n$ is obtained as a composition of mappings of the following type).
Then
\[
\Ff(n+1)=\left\{
\begin{array}{l}
\{F_0\subset \ldots\subset F_k\subset F_k\oplus\span{e}\subset \ldots\subset F_s\oplus\span{e}\} \\[1mm]
\parbox{10cm}{for some $k\leq s$, if $\hat{d}(\Ff,n+1)$ is a longer sequence than $\hat{d}(\Ff,n)$,} \\[3mm]
\{F_0\subset \ldots\subset F_k\subset F_{k+1}\oplus\span{e}\subset \ldots\subset F_s\oplus\span{e}\} \\[1mm]
\parbox{10cm}{for some $k< s$, if $\hat{d}(\Ff,n+1)$ is a sequence of the same length as $\hat{d}(\Ff,n)$.}
\end{array}
\right.
\]
The map $i_n$ is now defined via the respective formula
\[
i_n(\{M_k\}_{k=0}^s)=\left\{
\begin{array}{ll}
\{M_0\subset\ldots\subset M_k\subset M_k\oplus\span{e}\subset \ldots\subset M_s\oplus\span{e}\}, \\[2mm]
\{M_0\subset\ldots\subset M_k\subset M_{k+1}\oplus\span{e}\subset\ldots\subset M_s\oplus\span{e}\}.
\end{array}
\right.
\]
The embeddings $i_n$ have been introduced in \cite{Dimitrov-Penkov} in different notation.
In type (BCD) the construction is similar.
\end{remark}

\section{A characterization of large splitting parabolic subgroups}

\label{section-3}

Proposition \ref{P1} is incorporated in the following more complete statement.

\begin{proposition}
\label{P1bis}
Let $\P$ be a splitting parabolic subgroup of $\G$. Let $\Ff$ be the unique generalized flag for which $\P=\Stab_\G(\Ff)$ (see Proposition \ref{P0}). The following conditions are equivalent:
\begin{itemize}
\item[\rm (i)] $\P$ is large;
\item[\rm (ii)] $\Ff$ is weakly $E$-compatible with every admissible basis $E$;
\item[\rm (iii)] for every admissible basis $E$, the ind-variety $\G/\P=\G\cdot\Ff$ contains a generalized flag which is $E$-compatible;
\item[\rm (iv)] for every splitting Cartan subgroup $\H\subset\G$, there is $g\in\G$ such that $g\H g^{-1}\subset \P$.
\end{itemize}
\end{proposition}

\begin{proof}
{\rm (i)}$\Rightarrow${\rm (ii)}: 
Assume that $\P$ is a large splitting parabolic subgroup. Then the generalized flag $\Ff$ has the form
\[\Ff=\{F_0=0\subset F_1\subset F_2\subset \ldots \subset F_m=V\},\]
and moreover, there is a unique index $k\in\{1,\ldots,m\}$ for which the space $F_k/F_{k-1}$ is infinite dimensional. The subspace $F_{k-1}$ is finite dimensional and the subspace $F_k$ is finite codimensional. 
Moreover, since the generalized flag $\Ff$ is compatible with an admissible basis (Proposition \ref{P0}),  there is a finite subset $\Phi\subset V_*$ such that $F_k=\bigcap_{\phi\in\Phi}\ker\phi$.

Let $E\subset V$ be an admissible basis, hence the dual family $E^*=\{\phi_e:e\in E\}$ spans the subspace $V_*\subset V^*$. We can find a finite subset $I\subset E$ such that
\begin{itemize}
\item $F_{k-1}\subset V_1:=\langle I\rangle$;
\item $\Phi\subset\langle \phi_e:e\in I\rangle$, so that $V_2:=\langle E\setminus I\rangle\subset F_k$.
\end{itemize}
In type (BCD), up to choosing $I$ larger if necessary, we may assume that $I$ is stable by the involution $i_E:E\to E$ and contains the fixed point of $i_E$ if it exists. Hence the restriction of $\omega$
to $V_1$ and $V_2$ is nondegenerate and the decomposition
\[
V=V_1\oplus V_2
\]
is orthogonal in this case.

The sequence
\[
\Ff':=\{F'_0\subset F'_1\subset F'_2\subset\cdots\subset F'_m\}\quad\mbox{with}\quad F'_j:=V_1\cap F_j
\]
is a flag of the finite-dimensional space $V_1$, and we can find a basis $E_1$ of $V_1$ such that $\Ff'$ is compatible with $E_1$, that is, $F'_j$ is spanned by a subset $E_{1,j}\subset E_1$ for all $j$. Moreover, in type (BCD) we have
\[
(F'_j)^{\perp_1}=V_1\cap ({F_j}^\perp)=V_1\cap F_{m-j}=F'_{m-j}\quad\mbox{for all $j$,}
\]
where $\perp_1$ indicates the orthogonal space with respect to the restriction of $\omega$ to $V_1$, hence we may assume that the basis $E_1$ is isotropic with respect to the restriction of $\omega$.

We obtain a basis $E':=E_1\cup(E\setminus I)$ of $V$ which differs from $E$ by finitely many vectors, and is in addition admissible (see Remark \ref{R2.2}\,{\rm (a)}--{\rm (b)}).
For all $j$ we have
\[
F_j=\left\{
\begin{array}{ll}
F'_j=\langle E_{1,j}\rangle & \mbox{if $j\leq k-1$,} \\
V_2\oplus F'_j=\langle(E\setminus I)\cup E_{1,j}\rangle & \mbox{if $j\geq k$,}
\end{array}
\right.
\]
hence $\Ff$ is $E'$-compatible, and therefore weakly $E$-compatible.

{\rm (ii)}$\Rightarrow${\rm (iii)}: Assume that $\Ff$ is weakly $E$-compatible. This means that there is an admissible basis $E'$ which differs from $E$ by finitely many vectors, and such that
$\Ff$ is $E'$-compatible. 
The latter fact is equivalent to saying that $\Ff$ is a fixed point of the splitting Cartan subgroup $\H(E')$.
By Lemma \ref{L2.1}, we can find $g\in\G$ satisfying $\H(E)=g\H(E')g^{-1}$. 
This implies that the generalized flag $g\Ff$ is a fixed point of $\H(E)$, and therefore $g\Ff$ is $E$-compatible. 



{\rm (iii)}$\Rightarrow${\rm (iv)}: 
Every splitting Cartan subgroup of $\G$ is of the form
$\H(E)$ for an admissible basis $E$. By {\rm (iii)} we can find $g\in\G$
so that $g\Ff$ is $E$-compatible, that is, fixed by $\H(E)$. 
This yields $g^{-1}\H(E)g\subset\Stab_\G(\Ff)=\P$.

{\rm (iv)}$\Rightarrow${\rm (i)}: Assume that $\P$ is not large. 
This implies that one of the following cases occurs.
\begin{itemize}
\item[(a)] $\Ff$ contains a subspace $F_0$ that is both infinite dimensional and infinite codimensional, or
\item[(b)] $\Ff$ contains infinitely many subspaces, in particular, an increasing chain 
$F_1\subset F_2\subset\cdots\subset F_k\subset\cdots$
or a decreasing chain
$F_1\supset F_2\supset \cdots\supset F_k\supset\cdots$.
\end{itemize}
Moreover, in type (BCD), since we have $F^\perp\in\Ff$ whenever $F\in\Ff$, we may assume that $F_0$ is an isotropic subspace of $V$ in case (a), and that $\{F_k\}_{k\geq 1}$ is a chain of isotropic subspaces of $V$ in case (b).

Let $E$ be an admissible basis of $V$ such that the generalized flag $\Ff$ is $E$-compatible.
We claim that there is a double infinite sequence
\begin{equation}
\label{ene'n}
\{e_n,e'_n\}_{n\geq 1}\subset E
\end{equation}
such that
\begin{equation}
\label{Pene'n}
\forall n\geq 1,\ \exists F\in\Ff\ \mbox{such that}\  e_n\in F\ \mbox{and}\  e'_n\notin F
\end{equation}
and moreover
\begin{equation}
\label{Pene'nBCD}
\mbox{$\{e_n,e'_n\}_{n\geq 1}$ span an isotropic subspace of $V$}\quad\mbox{(in type (BCD)).}
\end{equation}
The construction of the double sequence $\{e_n,e'_n\}_{n\geq 1}$ can be done as follows. In case (b), for each $k\geq 2$ we take a vector $\varepsilon_k\in E$ lying in $F_k\setminus F_{k-1}$ in the case of an increasing chain,
respectively in $F_{k-1}\setminus F_{k}$ in the case of a decreasing chain. 
Then we set $(e_n,e'_n):=(\varepsilon_{2n},\varepsilon_{2n+1})$, respectively $(e_n,e'_n):=(\varepsilon_{2n+1},\varepsilon_{2n})$, 
and we have (\ref{Pene'n}). In type (BCD), since each subspace $F_k$ is isotropic, we get (\ref{Pene'nBCD}).

In case (a), in type (A), relying on the fact that $F_0$ is infinite dimensional and infinite codimensional, we take an infinite subset $\{e_n\}_{n\geq 1}\subset E$ of vectors which belong to $F_0$ and an infinite subset $\{e'_n\}_{n\geq 1}\subset E$ of vectors which do not belong to $F_0$. The so-obtained double sequence clearly satisfies (\ref{Pene'n}). In type (BCD), we first take an infinite subset $\{\varepsilon_k\}_{k\geq 1}\subset E$ of vectors which belong to $F_0$, then we set $e_n:=\varepsilon_{2n}$ and $e'_n:=i_E(\varepsilon_{2n+1})$. Since the subspace $F_0$ is isotropic, the vectors $e'_n$ do not belong to $F_0$, hence (\ref{Pene'n}) holds. Finally (\ref{Pene'nBCD}) holds due to the definition of the involution $i_E$. This completes the construction of the double sequence of (\ref{ene'n}).

Let $\H'\subset \G$ be the splitting Cartan subgroup associated to the double sequence $\{e_n,e'_n\}_{n\geq 1}$ as in Example \ref{E2}.
For every $g\in\G$, since
we have $g(e)\not=e$ for only finitely many $e\in E$, there is $n\geq 1$ such that $g(e_n)=e_n$ and $g(e'_n)=e'_n$. Then (\ref{PH'}) and (\ref{Pene'n})
yield a subspace $F\in\Ff$ and an element $h\in\H'$ such that
\[e_n\in F\quad\mbox{and}\quad ghg^{-1}(e_n)=e'_n\notin F,\quad\mbox{hence}\quad ghg^{-1}(F)\not=F.\]
This establishes that, for all $g\in\G$, we have $g\H'g^{-1}\not\subset\Stab_\G(\Ff)=\P$. The proof of the proposition is complete.
\end{proof}

\section{Exhaustion of $\X$ and a key lemma}

\label{section-4}

In this section we consider an ind-variety $\X$ of the form (\ref{X}). Our analysis is based on the following assumption: 
\begin{equation}
\label{all-but-one}
\mbox{the splitting parabolic subgroups $\P_1,\ldots,\P_{\ell-1}$ are large.}
\end{equation}

\subsection{Exhaustion}

\label{exhaustion}

Here we explain how to construct a natural exhaustion of the ind-variety $\X$. The notation introduced here is used in the subsequent sections.

For every $i\in\{1,\ldots,\ell\}$ there is a generalized flag $\Ff_i$ such that $\P_i=\Stab_\G(\Ff_i)$ (see Proposition \ref{P0}), hence $\X_i:=\G/\P_i=\G\cdot\Ff_i$. Let $E$ be an admissible basis of $V$ such that $\Ff_\ell$ is $E$-compatible. By Proposition \ref{P1bis}, for every $i\in\{1,\ldots,\ell-1\}$ we can find an element $h_i\in\G$ such that $h_i\Ff_i$ is $E$-compatible.
In view of the isomorphisms
\[
\G/\P_i\stackrel{\sim}{\to} \G/h_i\P_ih_i^{-1}\quad\mbox{for $1\leq i\leq \ell-1$},
\]
up to replacing $\P_i$ by $h_i\P_ih_i^{-1}=\Stab_\G(h_i\Ff_i)$, we may assume that
\[\mbox{$\Ff_1,\ldots,\Ff_\ell$ are $E$-compatible},\]
that is, 
\[\mbox{the splitting Cartan subgroup $\H(E)$ is contained in $\P_1,\ldots,\P_\ell$}.\]

Take a filtration of the basis 
\[E=\bigcup_{n\geq 1}E_n\]
by finite subsets (stabilized by $i_E$ in type (BCD)).
As in Sections \ref{section-1.1} and \ref{section-1.2}, we let
\[\G(E_n):=\G\cap \GL(\langle E_n\rangle)\quad\mbox{and}\quad \P_i(E_n):=\P_i\cap\GL(\langle E_n\rangle)
\]
which are respectively a classical algebraic group and a parabolic subgroup. In fact, since we are dealing with a single admissible basis $E$, it is harmless to avoid the reference to $E$ in the notation; we set for simplicity $\G(n):=\G(E_n)$ and $\P_i(n):=\P_i(E_n)$.

For each $i\in\{1,\ldots,\ell\}$, 
we follow the construction made in Section \ref{section-2.3}:
the generalized flag $\Ff_i$ gives rise to a flag in the finite-dimensional subspace $\langle E_n\rangle$, namely
let $\Ff_i(n)$ be the collection of subspaces
\[\Ff_i(n):=\{F\cap \span{E_n}\}_{F\in\Ff_i}.\]
The (finite-dimensional) algebraic variety 
\[
\X_i(n):=\G(n)/\P_i(n)=\G(n)\cdot\Ff_i(n)
\]
can be viewed as the set $\X_{\Ff_i}(n)$ of collections of subspaces of $\span{E_n}$ described in
Section \ref{section-2.3}.
(It is isomorphic in a natural way to a partial flag variety of the space $\span{E_n}$; see Remark \ref{R-section2.3}.)

For each $n\geq 1$, we have the embedding
\[
\phi_i(n):\X_i(n)\hookrightarrow\X_i(n+1),\ \{M_F\}_{F\in \Ff_i}\mapsto\{N_F\}_{F\in\Ff_i}
\]
given by
\[
N_F=M_F\oplus(F\cap \span{E_{n+1}\setminus E_n})\quad\mbox{for all $F\in\Ff_i$.}
\]
Finally, for each $i\in\{1,\ldots,\ell\}$, the ind-variety $\X_i=\G/\P_i=\G\cdot\Ff_i$ is obtained as the limit of the inductive system
\[
\X_i(1)\hookrightarrow\X_i(2)\hookrightarrow\cdots\hookrightarrow\X_i(n)\hookrightarrow\X_i(n+1)\hookrightarrow\cdots,
\]
thus we have an exhaustion
\[\X_i=\bigcup_{n\geq 1}\X_i(n).\] 
Altogether, we have the following exhaustion of the ind-variety $\X$:
\[\X=\bigcup_{n\geq 1}\X(n),\quad\X(n):=\X_1(n)\times\cdots\times\X_\ell(n),\]
where for each $n$ we consider the embedding
\begin{equation}
\label{phi(n)}
\phi(n):=\prod_{i=1}^\ell\phi_i(n):\X(n)\hookrightarrow\X(n+1).
\end{equation}

\begin{remark} 
The construction presented in this section is only possible in the case where the ind-varieties $\X_1,\ldots,\X_\ell$ have points $\Ff_1,\ldots,\Ff_\ell$ which admit a common compatible basis.\ Assumption (\ref{all-but-one}) (combined with Proposition \ref{P1bis}) is crucial in this respect.
\end{remark}

\subsection{Key lemma} 
\label{subsection-keylemma}

We still assume (\ref{all-but-one}).
Our key lemma is as follows.

\begin{lemma}
\label{L-key}
Let $\phi(n):\X(n)\hookrightarrow\X(n+1)$ be the embedding (\ref{phi(n)}).
Let $\Ff,\Ff'\in\X(n)$. Assume that $\phi(n)(\Ff)$ and $\phi(n)(\Ff')$ belong to the same $\G(n+1)$-orbit. Then $\Ff$ and $\Ff'$ belong to the same $\G(n)$-orbit.

Consequently, for every $n\geq 1$, the embedding $\X(n)\hookrightarrow \X$ induces an injection of orbit sets $\X(n)/\G(n)\hookrightarrow \X/\G$.
\end{lemma} 


In type (A), we may suppose that
\[E_{n+1}=E_n\cup\{e\}.\]
Hence $M:=\span{E_n}$ is a hyperplane of $N:=\span{E_{n+1}}=M\oplus\span{e}$.
Every element of the variety $\X(n)$ (respectively $\X(n+1)$) consists of a collection of subspaces
$\{M_i\}_{i\in I}$ of $M$ (respectively $\{N_i\}_{i\in I}$ of $N$), and the map
$\phi(n):\X(n)\to\X(n+1)$ is of the form
\[
\phi(n):\{M_i\}\mapsto\{N_i\}\quad\mbox{with}\quad N_i=\left\{\begin{array}{ll}
M_i & \mbox{if $i\in I_0$,} \\
M_i\oplus\span{e} & \mbox{if $i\in I\setminus I_0$,}
\end{array}\right.
\]
for some subset $I_0\subset I$.
Therefore, in type (A) Lemma \ref{L-key} follows from the following lemma from linear algebra.

\begin{lemma}
\label{L2}
Let $N$ be a finite-dimensional complex vector space, $M\subset N$ a hyperplane, $L\subset N$ a line such that $N=M\oplus L$.
Let $\{M_i\}_{i\in I}$ and $\{M'_i\}_{i\in I}$ be two collections of subspaces of $M$, indexed by an arbitrary set $I$. Given a subset $I_0\subset I$, assume that there exists a $g\in \GL(N)$ such that
\[g(M_i)=M'_i\ \ \forall i\in I_0\quad\mbox{and}\quad g(M_i\oplus L)=M'_i\oplus L\ \ \forall i\in I\setminus I_0.\]
Then there is $h\in \GL(M)$ satisfying
\[h(M_i)=M'_i\quad \forall i\in I.\]
\end{lemma}

\begin{proof}[Proof of Lemma \ref{L2}]
Let $p:N=M\oplus L\to M$ denote the linear projection.
First, assume that $g(L)=L$. Then the linear map
\[
h:=p\circ g|_M
\]
is an element of $\GL(M)$. For $i\in I_0$ we have
\[M'_i=p(M'_i)=p(g(M_i))=h(M_i),\]
and for $i\in I\setminus I_0$ we have
\[M'_i=p(M'_i\oplus L)=p(g(M_i\oplus L))=p(g(M_i)\oplus L)=p(g(M_i))=h(M_i).\]
The lemma is proved in this case. 

Thus it remains to consider the case where
\[g(L)\not=L.\]
Then $K:=L+g(L)$ is a $2$-dimensional subspace of $N$.

The sums 
\[M_0:=\sum_{i\in I_0}M_i \qquad\mbox{and}\qquad M'_0:=\sum_{i\in I_0}M'_i\]
are subspaces of $M$, and by the properties of $g$ we have $g(M_0)=M'_0$, so $\dim M_0=\dim M'_0$.
Hence we can find $h_0\in\GL(M)$ such that $h_0(M'_0)=M_0$. Let $h':=h_0\oplus\mathrm{id}_L\in\GL(M\oplus L)=\GL(N)$. For every $i\in I$ we set $M''_i:=h_0(M'_i)$. Then
\[
h'(g(M_i))=h'(M'_i)=M''_i\ \ \forall i\in I_0
\]
and
\[
h'(g(M_i\oplus L))=h'(M'_i\oplus L)=M''_i\oplus L\ \ \forall i\in I\setminus I_0.
\]
Moreover, $h'(g(M_0))=M_0$. Therefore, by dealing with $\{M''_i\}_{i\in I}$ instead of $\{M'_i\}_{i\in I}$, we can assume that 
\[M_0=M'_0,\quad \mbox{i.e.,}\quad g(M_0)=M_0.\]
Let $e\in L$, $e\not=0$.
We distinguish two cases.

\medskip
\noindent
{\it Case 1:} $K\cap M_0=0$.

Let $\{e_1,\ldots,e_r\}$ be a basis of $M_0$. Then the vectors $e_1,\ldots,e_r,e,g(e)$ are linearly independent, and we can find vectors $e_{r+3},\ldots,e_d$ such that
\[\{e_1,\ldots,e_r,e,g(e),e_{r+3},\ldots,e_d\}\mbox{ is a basis of $N$.}\]
Let $\eta\in\GL(N)$
satisfy $\eta(e_k)=e_k$ for all $k\in\{1,\ldots,r,r+3,\ldots,d\}$, $\eta(e)=g(e)$, and $\eta(g(e))=e$. In particular, $\eta|_{M_0}=\mathrm{id}_{M_0}$, hence \[\eta(g(M_i))=\eta(M'_i)=M'_i\]
for all $i\in I_0$.
Note also that
\begin{equation}
\label{eta}
N=K\oplus \span{e_1,\ldots,e_r,e_{r+3},\ldots,e_d}=K+\ker(\eta-\mathrm{id}_{N}).
\end{equation} 
For all $i\in I\setminus I_0$, we have $M'_i\oplus L=g(M_i\oplus L)$, therefore the subspace $M'_i\oplus L$ contains $L+g(L)=K$. By (\ref{eta}), and since $\eta(K)=K$, this implies that the subspace $M'_i\oplus L$ is $\eta$-stable. Hence,
\[\eta(g(M_i\oplus L))=\eta(M'_i\oplus L)=M'_i\oplus L.\]
Since $\eta(g(e))=e$, i.e. $\eta\circ g(L)=L$, this brings us back to the situation treated at the beginning of the proof.

\medskip
\noindent
{\it Case 2:} $K\cap M_0\not=0$.

As $L\not\subset M_0$, we have $K\not\subset M_0$, and consequently $\dim K\cap M_0=1$ in this case.
Note also that $g(e)\notin g(M_0)=M_0$. Hence $K\cap M_0=\span{e_1}$ for some vector $e_1$ which also satisfies $K=\span{e,e_1}=\span{g(e),e_1}$. Let $\{e_1,\ldots,e_r\}$ be a basis of $M_0$ containing the vector $e_1$. Then $\{e_1,\ldots,e_r,e\}$ and $\{e_1,\ldots,e_r,g(e)\}$ are two bases of $M_0+K$. We extend them into bases of $N$ by adding a common set of vectors $e_{r+2},\ldots,e_{d}$. Let $\eta\in\mathrm{GL}(N)$ satisfy $\eta(e_k)=e_k$ for all $k$ and $\eta(g(e))=e$. Then $\eta|_{M_0}=\mathrm{id}_{M_0}$, and this implies
\[\eta(g(M_i))=\eta(M'_i)=M'_i\ \ \forall i\in I_0.\]
Moreover, $\eta$ satisfies
\begin{equation}
\label{etabis}
N=K\oplus \span{e_2,\ldots,e_r,e_{r+2},\ldots,e_{d}}=K+\ker (\eta-\mathrm{id}_{N}).
\end{equation}
For every $i\in I\setminus I_0$,
we have $K\subset M'_i\oplus L=g(M_i\oplus L)$. In view of (\ref{etabis}), and since $\eta(K)=K$, we deduce that
the subspace $M'_i\oplus L$ is $\eta$-stable, hence
\[\eta(g(M_i\oplus L))=\eta(M'_i\oplus L)=M'_i\oplus L\ \ \forall i\in I\setminus I_0.\]
Since $\eta(g(L))=L$, again we are brought back to the situation already treated at the beginning of the proof. The proof of the lemma is now complete.
\end{proof}


In type (BCD), with the notation of Lemma \ref{L-key}, we have already the implications
\begin{eqnarray*}
& & \mbox{$\phi(n)(\Ff)$ and $\phi(n)(\Ff')$ belong to the same $\G(n+1)$-orbit} \\
& \Rightarrow & \mbox{$\phi(n)(\Ff)$ and $\phi(n)(\Ff')$ belong to the same $\GL(\span{E_{n+1}})$-orbit} \\
& \Rightarrow & \mbox{$\Ff$ and $\Ff'$ belong to the same $\GL(\span{E_n})$-orbit,}
\end{eqnarray*}
where the last implication is valid since we have already proved Lemma \ref{L-key} in type (A).
For completing the proof of Lemma \ref{L-key} in type (BCD), we have to show that $\Ff$ and $\Ff'$ belong to the same $\G(n)$-orbit.
This conclusion is deduced from the following general fact.

\begin{lemma}
\label{L3}
Let $M$ be a finite-dimensional linear space, endowed with a nondegenerate orthogonal or symplectic bilinear form $\omega$.
We consider the group 
\[G(M,\omega)=\{g\in\GL(M):\mbox{$g$ preserves $\omega$}\}.\]
Let 
$I$ be a set equipped with an involution $i\mapsto i^*$, and let
$\Ff=\{M_i\}_{i\in I}$ and $\Ff'=\{M'_i\}_{i\in I}$ be two collections of subspaces satisfying
\begin{eqnarray*}
 & d_i:=\dim M_i=\dim M'_i\quad\mbox{for all $i\in I$}, \\
 & M_i^\perp =M_{i^*}\in \Ff\quad\mbox{and}\quad M_i'^\perp=M'_{i^*}\in\Ff'\quad\mbox{for all $i\in I$}.
\end{eqnarray*}
Assume that there is $g\in \GL(M)$ with $g(M_i)=M'_i$ for all $i$. Then there is $h\in G(M,\omega)$ with $h(M_i)=M'_i$ for all $i$.
\end{lemma}

\begin{proof}
We define $X$ to be the set of collections of subspaces $\{M_i\}_{i\in I}$ with  $\dim M_i=d_i$ for all $i$, and we consider the action of $G:=\GL(M)$ on $X$ given by
\[g\cdot \{M_i\}_{i\in I}=\{g(M_i)\}_{i\in I}.\]
Note that $X$ is endowed with the involution
\[\sigma:X\to X,\ \{M_i\}_{i\in I}\mapsto \{M_{i^*}^\perp\}_{i\in I}.\]
Let $X^\sigma$ be the fixed point set of this involution. Then $\mathcal{F}$ and $\Ff'$ are elements of $X^\sigma$.

Let $u^*$  denote the adjoint morphism of an endomorphism $u\in\mathrm{End}(M)$ with respect to the form $\omega$. Thus $G$ is also endowed with an involution given by
\[G\to G,\ g\mapsto g^\sigma:=(g^*)^{-1},\]
and $G(M,\omega)$ coincides with the subgroup $G^\sigma$ of fixed points of this involution. Then the claim made in the statement follows once we show that two elements of $X^\sigma$ are $G^\sigma$-conjugate whenever they are $G$-conjugate. This is exactly \cite[Proposition 2.1]{MWZ2} (conditions (1)--(3) of \cite[Proposition 2.1]{MWZ2} are clearly verified).
\end{proof}

\section{Proof of the direct implication in Theorem \ref{T1}}

\label{section-5}
Arguing indirectly, assume that $(\P_1,\P_2)$ is a pair of splitting parabolic subgroups which does not satisfy the condition of Theorem \ref{T1}, namely, up to exchanging the roles of $\P_1$ and $\P_2$, we may assume that
\begin{description}
\item[Case 1] $\P_1$ is not semilarge, or
\item[Case 2] $\P_1,\P_2$ are semilarge but not large.
\end{description}
Considering generalized flags $\Ff_1,\Ff_2$ such that $\P_1=\Stab_\G(\Ff_1)$ and $\P_2=\Stab_\G(\Ff_2)$, the condition of Case 1 means that
\begin{equation}
\label{5.1}
\mbox{$\Ff_1$ has an infinite number of subspaces.}
\end{equation}
The condition of Case 2 implies that
\begin{equation}
\label{5.2}
\parbox{12cm}{in each generalized flag $\Ff_1$ and $\Ff_2$, there is at least one subspace which is both infinite dimensional and infinite codimensional.}
\end{equation}

We will show that $\X=\G/\P_1\times \G/\P_2$ has infinitely many $\G$-orbits.
Since the map
\[\P_1\cdot(g\P_2)\mapsto \G\cdot(\P_1,g\P_2)\]
is a bijection
between the set of $\P_1$-orbits on $\G/\P_2$ and the set of $\G$-orbits on $\X$, it suffices to show that $\G/\P_2$ has infinitely many $\P_1$-orbits.

Let $F_2\in\Ff_2$ be such that $0\subsetneq F_2\subsetneq V$. In Case 2, by virtue of (\ref{5.2}), we assume that $F_2$ is infinite dimensional and infinite codimensional. In type (BCD) we assume that $F_2\subset F_2^\perp$. Recall
from Example \ref{E1}\,{\rm (b)} that the generalized flag associated to $F_2$ is given by
\[
\Ff_{F_2}=\left\{
\begin{array}{ll}
\{0\subset F_2\subset V\} & \mbox{in type (A),} \\
\{0\subset F_2\subset F_2^\perp\subset V\} & \mbox{in type (BCD).} \\
\end{array}
\right.
\]
By replacing the parabolic subgroup $\P_2$ by the larger splitting parabolic subgroup $\hat\P_2:=\Stab_\G(\mathcal{F}_{F_2})$, we may assume that $\Ff_2=\Ff_{F_2}$.

We fix an admissible basis $E$ such that $\Ff_2$ is $E$-compatible.
Then the ind-variety $\G/\P_2=\G\cdot\Ff_2$ consists of generalized flags which are $E$-commensurable
with $\Ff_2$, in particular are of the form $\Ff_F$ for $F\subset V$.
Our aim is to construct an infinite sequence of such generalized flags which belong to pairwise distinct $\P_1$-orbits.
By a slight abuse of terminology, we say that a subspace $F$ is weakly $E$-compatible if its associated generalized flag $\Ff_F$ is weakly $E$-compatible. Also, we say that $F'$ is $E$-commensurable with $F$ if $\Ff_{F'}$ is $E$-commensurable
with $\Ff_F$. 

\begin{lemma}
\label{L4-T1}
Let $F$ be weakly $E$-compatible (with $F\subset F^\perp$ in type (BCD)). Let $\phi\in V_*$ satisfy $F\not\subset\ker \phi$. Fix $v\in V\setminus F$. 
In type (BCD) we assume in addition that the vector $v$ is isotropic and belongs to $(F\cap\ker\phi)^\perp$.
Then $F':=(F\cap \ker\phi)\oplus\C v$ is $E$-commensurable with $F$.
\end{lemma} 

\begin{proof}
Clearly, a subspace 
is weakly $E$-compatible if and only if it has a finite-codimensional subspace spanned by a subset of $E$. 
Let $I\subset E$ be such that $F$ contains $\span{I}$ as a finite-codimensional subspace.
There is a finite set $J\subset E$ with $\phi\in\span{\phi_e:e\in J}$. Then $F'$ contains $\span{I\setminus J}$ as a finite-codimensional subspace. Hence $F'$ is weakly $E$-compatible.

Let a vector $v'$ satisfy $F=(F\cap \ker \phi)\oplus\mathbb{C}v'$. 
%
Then, we see that $\Ff_F$ and $\Ff_{F'}$ are $E$-commensurable by considering
any finite-dimensional subspace $U\subset V$ such that $v,v'\in U$ in type (A) and which satisfies in addition
$(F\cap\ker\phi)^\perp\cap U\not\subset\span{v}^\perp$
and $(F\cap\ker\phi)^\perp\cap U\not\subset\span{v'}^\perp$ in type (BCD).
\end{proof}

\begin{lemma}
\label{lemma-2-T1}
Let $L\subsetneq M$ and $F\not=0$ be subspaces of $V$, 
with $F$ weakly $E$-compatible. In type (BCD) we assume that these subspaces are isotropic.
In type (A) we assume that $L,M$ are of the form
\[
L=\bigcap_{\phi\in\Phi}\ker\phi\subsetneq M=\bigcap_{\phi\in\Psi}\ker\phi
\]
for subsets $\Psi\subsetneq \Phi\subset V_*$.
\begin{itemize}
\item[\rm (a)] If $M\not\subset F$ and $F\not\subset M$, then there is a subspace $F'\subset V$ which is $E$-commensurable with $F$, isotropic in type (BCD), and such that
\begin{eqnarray*}
 & \mbox{$F\cap M$ is a hyperplane of $F'\cap M$} \\
 & \mbox{and $F'/F'\cap M$ embeds as a hyperplane in $F/F\cap M$.}
\end{eqnarray*}
\item[\rm (b)] If $F\cap M\not\subset L$ and $F+M\not=V$, then there is a subspace $F'\subset V$ which is
$E$-commensurable with $F$, isotropic in type (BCD), and such that
\[
F'\cap L=F\cap L,\quad  \mbox{$F'\cap M$ is a hyperplane of $F\cap M$.}
\]
\item[\rm (c)] If $F\subset L$, then there is a subspace $F'\subset V$ which is
$E$-commensurable with $F$
and such that
$F'\not\subset L$, $F'\subset M$.
\end{itemize}
\end{lemma}

\begin{proof}
{\rm (a)} Since $M\not\subset F$ we can find a vector $v\in M\setminus F$.
In type (A), since $F\not\subset M$, we can find $\phi\in\Psi$ such that $F\not\subset \ker\phi$.
In type (BCD), either $v\notin F^\perp$ and we take $\phi=\omega(v,\cdot)$,
or $v\in F^\perp$ and we take $\phi=\omega(v',\cdot)$ for any $v'\in M^\perp\setminus F^\perp$.
In all the cases, it follows from Lemma \ref{L4-T1} that the subspace
\[
F':=(F\cap \ker\phi)\oplus\span{v}
\]
is $E$-commensurable, isotropic in type (BCD), and satisfies
\[
F'\cap M=F\cap M\oplus\span{v}.
\]
Moreover, the linear projection $F'\to F\cap\ker\phi$ yields an isomorphism  $F'/F'\cap M\stackrel{\sim}{\to} F\cap\ker\phi/F\cap M$.

{\rm (b)} 
In type (A) we take a vector $v\in V\setminus (F+M)$ and we find $\phi\in\Phi$ such that $F\cap M\not\subset \ker\phi$.
By Lemma \ref{L4-T1}, the subspace
\begin{equation}
\label{F'}
F':=(F\cap\ker\phi)\oplus\span{v}
\end{equation}
is $E$-commensurable and satisfies
\[
F'\cap L=F\cap L,\quad F'\cap M=(F\cap M)\cap\ker\phi.
\]

In type (BCD) the construction is adapted as follows.
Since $F\cap M\not\subset L$, there is a vector $v'\in L^\perp$ such that $v'\notin(F\cap M)^\perp$,
that is, $F\cap M\not\subset\span{v'}^\perp$. We take $\phi:=\omega(v',\cdot)$.
The fact that $F\cap M\not\subset \span{v'}^\perp$ also implies that there is an isotropic vector of the
form $v=v'+w$ with $w\in F\cap M$.
Since $F\cap M\subset (F+M)^\perp$, we have $v'\notin F+M$, and hence $v\notin F+M$.
Moreover, 
$v\in F^\perp+\span{v'}=(F\cap \ker\phi)^\perp$.
By Lemma \ref{L4-T1}, the subspace $F'$ of (\ref{F'}) corresponding to this choice of $\phi$ and $v$ is isotropic and fulfills the required conditions.

{\rm (c)} This time the conditions are fulfilled by the subspace $F':=(F\cap\ker\phi)\oplus\span{v}$, where $v\in M\setminus L$ and $\phi$ is any element of $V_*$ such that $F\not\subset\ker\phi$.
\end{proof}

{\bf Claim 1:} If there is $F_1\in\Ff_1$ (isotropic in type (BCD)) such that $F_1\cap F_2$ has infinite codimension in $F_1$ and $F_2$,
then there is a sequence $\{F^{(n)}\}_{n\geq 0}$ of (respectively isotropic) subspaces which are $E$-commensurable with $F_2$ and satisfy
\[F_1\cap F^{(0)}\subsetneq F_1\cap F^{(1)}\subsetneq\cdots\subsetneq F_1\cap F^{(n)}\subsetneq\cdots.\]

\begin{proof}[Proof of Claim 1]
The sequence $\{F^{(n)}\}_{n\geq 0}$ is constructed by induction: we let $F^{(0)}=F_2$ and,
once $F^{(n)}$ is defined, we let $F^{(n+1)}$ be the subspace $F'$ obtained by applying Lemma \ref{lemma-2-T1}\,{\rm (a)} with $F=F^{(n)}$ and $M=F_1$ (each time $F_1\cap F^{(n)}$ still has infinite codimension in $F_1$ and $F^{(n)}$ hence the conditions
for applying Lemma \ref{lemma-2-T1}\,{\rm (a)} are 
fulfilled).
\end{proof}

{\bf Claim 2:} If $F_1\in\Ff_1$ (isotropic in type (BCD)) is such that $F_1\cap F_2$ has finite codimension in $F_1$ or $F_2$,
then there is a (respectively isotropic) subspace $F$ which is $E$-commensurable with $F_2$ and such that 
$F\subset F_1$ or $F_1\subset F$.

\begin{proof}[Proof of Claim 2]
Let $m_i:=\mathrm{codim}_{F_i}F_1\cap F_2$ and $m:=\min\{m_1,m_2\}$. By applying $m$ times Lemma \ref{lemma-2-T1}\,{\rm (a)} in the same way as in the proof of Claim 1, we obtain a (respectively isotropic) subspace $F$ which is $E$-commensurable with $F_2$ and such that
\[\mathrm{codim}_{F_1}F_1\cap F=0\quad\mbox{or}\quad\mathrm{codim}_{F}F_1\cap F=0,\]
that is, $F_1\subset F$ or $F\subset F_1$.
\end{proof}

{\bf Case 1} can now be adressed as follows. If there is $F_1\in\Ff_1$ (respectively isotropic) such that $F_1\cap F_2$ has infinite codimension in $F_1$ and $F_2$, then Claim 1 yields an infinite sequence $\{\Ff_{F^{(n)}}\}_{n\geq 0}$ of elements of $\G/\P_2$ which belong to pairwise distinct orbits of $\P_1$. Indeed, if there were $n<m$ with $g(\Ff_{F^{(n)}})=\Ff_{F^{(m)}}$ for some $g\in\P_1=\Stab_\G(\Ff_1)$, then we would have $g(F_1\cap F^{(n)})=F_1\cap F^{(m)}$, in contradiction with Lemma \ref{L-new} as the inclusion $F_1\cap F^{(n)}\subsetneq F_1\cap F^{(m)}$ is strict.

It remains to consider the case where $F_1\cap F_2$ has finite codimension in $F_1$ or $F_2$ for all (respectively isotropic) $F_1\in\Ff_1$. Invoking Claim 2, and using that $\Ff_1$ contains an infinite number of subspaces (see (\ref{5.1})), for all $n\geq 1$ we can find a sequence
\[
F_{1,0}\subsetneq F_{1,1}\subsetneq\cdots\subsetneq F_{1,n}
\]
of subspaces of $\Ff_1$ such that
\[
\forall k\in\{0,\ldots,n\},\ \exists \Ff_F\in\G/\P_2\ \mbox{such that}\ F_{1,k}\subset F
\]
or
\[
\forall k\in\{0,\ldots,n\},\ \exists \Ff_F\in\G/\P_2\ \mbox{such that}\ F\subset F_{1,k}.
\]
In the former case, by Lemma \ref{lemma-2-T1}\,{\rm (b)}, for each $k\in\{1,\ldots,n\}$ there
is a (respectively isotropic) subspace $F^{(k)}$, $E$-commensurable with $F_2$, such that $F_{1,k-1}\subset F^{(k)}$, $F_{1,k}\not\subset F^{(k)}$.
In the latter case, invoking Lemma \ref{lemma-2-T1}\,{\rm (c)}, for each $k\in\{1,\ldots,n\}$ we find this time $F^{(k)}$ such that $F^{(k)}\subset F_{1,k}$, $F^{(k)}\not\subset F_{1,k-1}$.
In both cases, the so obtained generalized flags $\Ff_{F^{(1)}},\ldots,\Ff_{F^{(n)}}\in \G/\P_2$ belong to pairwise distinct orbits of $\P_1$. Since $n$ is arbitrarily large, we conclude that there are infinitely many $\P_1$-orbits in $\G/\P_2$.

In {\bf Case 2}, we assume that $F_2\in\Ff_2$ has infinite dimension and infinite codimension in $V$,
and that there is $F_1\in\Ff_1$ with the same property (see (\ref{5.2})).
In type (BCD) these subspaces are also assumed to be isotropic. In the case where $F_1\cap F_2$ has infinite codimension in $F_1$ and $F_2$, Claim 1 yields infinitely many elements $\Ff_{F^{(n)}}\in\G/\P_2$ which belong to pairwise distinct $\P_1$-orbits.

It remains to consider the case where $F_1\cap F_2$ has finite codimension in $F_1$ or $F_2$. This implies that $F_1\cap F$ has infinite dimension and $F_1+F$ has infinite codimension in $V$ whenever $F$ is $E$-commensurable with $F_2$. 
By applying Lemma \ref{lemma-2-T1}\,{\rm (b)} with $(L,M)=(0,F_1)$, we get a sequence $\{F^{(n)}\}_{n\geq 0}$ of (respectively isotropic) subspaces which are $E$-commensurable with $F_2$ and satisfy
\[F_1\cap F^{(0)}\supsetneq F_1\cap F^{(1)}\supsetneq\cdots\supsetneq F_1\cap F^{(n)}\supsetneq\cdots.\]
Therefore, the associated generalized flags $\Ff_{F^{(n)}}$ (for $n\geq 0$) are points of $\G/\P_2$ that belong to pairwise distinct $\P_1$-orbits. We again conclude that $\G/\P_2$ has an infinite number of $\P_1$-orbits. The proof of the direct implication in Theorem \ref{T1} is now complete.

\section{Proof of the inverse implication in Theorem \ref{T1}}

\label{section-6}

We assume that $\P_1$ is large and $\P_2$ is semilarge. Hence assumption (\ref{all-but-one}) is fulfilled, and we can find an exhaustion
\[\X=\G/\P_1\times \G/\P_2=\bigcup_{n\geq 1} \X(n)\]
as in Section \ref{exhaustion}. By Lemma \ref{L-key}, we have inclusions of orbit sets
\[
\X(n)/\G(n)\hookrightarrow\X(n+1)/\G(n+1)\ \mbox{ for all $n\geq 1$}
\]
and the orbit set $\X/\G$ is the direct limit
\[
\X/\G=\lim_\to \X(n)/\G(n).
\]
To show that $\X$ has a finite number of $\G$-orbits, it is sufficient to estimate the number $s_n$ of $\G(n)$-orbits on $\X(n)$ and prove that the sequence $\{s_n\}_{n\geq 1}$ is bounded.

In type (BCD), $\X(n)$ is the set of ordered pairs of isotropic flags of a given type, in a finite-dimensional space $M$ endowed with the nondegenerate bilinear form $\omega$, and $\G(n)$ is the group $G(M,\omega)$ of transformations which preserve $\omega$. According to Lemma \ref{L3}, we have
$s_n\leq s_n^\mathrm{A}$ where $s_n^\mathrm{A}$ stands for the number of $\GL(M)$-orbits on the set of ordered pairs of (not necessarily isotropic) flags of the same type. Thus for showing that the sequence $\{s_n\}_{n\geq 1}$ is bounded, it suffices to show that the sequence $\{s_n^\mathrm{A}\}_{n\geq 1}$ is bounded. Therefore, it is enough to deal with the type-(A) case.

Since $\P_1$ is large, the generalized flag $\Ff_1$ is a finite chain
\[
\Ff_1=\{F_{1,0}=0\subset F_{1,1}\subset\ldots\subset F_{1,p_0-1}\subset F_{1,p_0}\subset\ldots\subset F_{1,p}=V\},
\]
such that $c_{1,k}:=\dim F_{1,k}/F_{1,k-1}$ is finite for all $k\not=p_0$.
The generalized flag $\Ff_2$ is a finite chain as $\P_2$ is semilarge. For $n\geq 1$ large, $\X(n)$ is a double flag variety of the form
\[
X:=\Fl(c_1,\ldots,c_p)\times \Fl(d_1,\ldots,d_q),
\]
whose elements are ordered pairs of flags
\[(F_0=0\subset F_1\subset\ldots\subset F_p=\span{E_n},F'_0=0\subset F'_1\subset \ldots\subset F'_q=\span{E_n})\]
such that $\dim F_k/F_{k-1}=c_k$ and $\dim F'_\ell/F'_{\ell-1}=d_\ell$ for all $k,\ell$.
We have $c_k>0$, $d_\ell>0$ and
\[c_1+\ldots+c_p=d_1+\ldots+d_q=m:=\dim\span{E_n}.\]
In addition, by choosing $n$ large, we may assume that  $c_k=c_{1,k}$ for all $k\not=p_0$, and thus $c_{p_0}=m-\sum_{k\not=p_0}c_{1,k}$. In particular, $c_k$ is independent of $n$ for all $k\not=p_0$.
We must show that the number $s_n$ of $G:=\GL(\span{E_n})$-orbits
on $X$ can be bounded by a constant which depends only on the numbers $c_k$ (for $k\not=p_0$), $p$, and $q$.
By the Bruhat decomposition, $s_n$ is the cardinality of the double coset
\[
\mathfrak{S}_{c_1}\times\cdots\times\mathfrak{S}_{c_p}\backslash\mathfrak{S}_m/\mathfrak{S}_{d_1}\times\cdots\times\mathfrak{S}_{d_q}.
\]

An element of the quotient $\mathfrak{S}_m/\mathfrak{S}_{d_1}\times\cdots\times\mathfrak{S}_{d_q}$ can be viewed as a map $\tau:\{1,\ldots,m\}\to\{1,\ldots,q\}$ such that $\tau^{-1}(j)$ has $d_j$ elements for all $j$. Every such map $\tau$ belongs to the 
$\mathfrak{S}_{c_1}\times\cdots\times\mathfrak{S}_{c_p}$-orbit of a map $\tau_0$ which is in addition nondecreasing on each interval $[\bar{c}_{k-1}+1,\bar{c}_k]$ with $\bar{c}_k:=c_1+\ldots+c_k$.
Such a map $\tau_0$ is completely determined by its restriction to $\{i\in[1,m]:i\leq \bar{c}_{p_0-1}\mbox{ or }i>\bar{c}_{p_0}\}$.
This restriction is a map
\[\{1,\ldots,\bar{c}_{p_0-1}\}\cup\{\bar{c}_{p_0}+1,\ldots,m\}\to \{1,\ldots,q\}\]
where the set on the left-hand side has $C:=\sum_{k\not=p_0}c_k$ elements. There are $q^C$ maps between these two sets. Therefore, we conclude that
\[s_n=|\mathfrak{S}_{c_1}\times\cdots\times\mathfrak{S}_{c_p}\backslash\mathfrak{S}_m/\mathfrak{S}_{d_1}\times\cdots\times\mathfrak{S}_{d_q}|\leq q^C.\]
The proof of the theorem is complete.

\section{Proof of Theorem \ref{T2}}

\label{section-7}

The last assertion in Theorem \ref{T2} is a consequence of the first part of the statement combined with the classification of triple flag varieties of finite type for classical groups established in the references \cite{MWZ1,MWZ2,Matsuki1,Matsuki2}. Details are given in Appendix \ref{appendix}. In this section we prove the first part of the theorem.


Let $\X=\X_1\times \X_2\times \X_3$, where each factor $\X_i=\G/\P_i$ ($i=1,2,3$) is the quotient by a splitting parabolic subgroup $\P_i\subset\G$ and can be viewed as an ind-variety of generalized flags (Section \ref{section-2.3}). Evidently, $\X$ can have a finite number of $\G$-orbits only if every product of two factors $\X_i\times \X_j$ ($1\leq i<j\leq 3$) has a finite number of $\G$-orbits. In view of Theorem \ref{T1}, this property holds only if all three parabolic subgroups $\P_1,\P_2,\P_3$ are semilarge and two of them are large. This justifies the first claim in Theorem \ref{T2}. 

In what follows, we assume that 
\[\mbox{$\P_1,\P_2$ are large and $\P_3$ is semilarge}.\]
In this way, condition (\ref{all-but-one}) is satisfied and we may consider the construction of Section \ref{exhaustion}. Namely, we may choose an admissible basis $E$ such that $\X_1,\X_2$, and $\X_3$ contain an element which is $E$-compatible.
Relying on a filtration $E=\bigcup_{n\geq 1} E_n$ as in Section \ref{exhaustion}, we get exhaustions
\[\G=\bigcup_{n\geq 1}\G(n)\quad\mbox{and}\quad\X=\bigcup_{n\geq 1} \X_1(n)\times \X_2(n)\times \X_3(n)\]
such that $\X_i(n)$ (for $i=1,2,3$) is a (finite-dimensional) flag variety for the algebraic group $\G(n)$.

We have to show that $\X$ has a finite number of $\G$-orbits if and only if $\X_1(n)\times \X_2(n)\times \X_3(n)$ has a finite number of $\G(n)$-orbits for all $n$. The direct implication follows from Lemma \ref{L-key}. In the rest of this section, we assume that
\begin{equation}
\label{8}
\mbox{$\X_1(n)\times \X_2(n)\times \X_3(n)$ has finitely many $\G(n)$-orbits for all $n$,}
\end{equation}
and we need to show that $\X$ has finitely many $\G$-orbits.

For every $n\geq 1$, we denote $V_n=\span{E_n}$ and $V'_n=\span{E\setminus E_n}$. Then
\[V=V_n\oplus V'_n.\]
In type (BCD), the subspaces $V_n$ and $V'_n$ are orthogonal with respect to the form $\omega$. Note that the restrictions of $\omega$ to $V_n$ and $V'_n$, respectively denoted by $\omega_n$ and $\omega'_n$, are nondegenerate. Let $\pi_n:V\to V_n$ and $\pi'_n:V\to V'_n$ be the projections determined by  the above decomposition.

Let $\G(V'_n)$ stand for the subgroup of elements $g\in\G$ such that 
$g(V'_n)= V'_n$ and $g(e)=e$ for all $e\in E_n$. Note that $\G(V'_n)$ can be viewed as a subgroup of $\GL(V'_n)$, and it is an ind-group of the same type as $\G$. 

We point out two preliminary facts.

\begin{lemma}
\label{L6}
Let $\P\subset \G$ be a large splitting parabolic subgroup. 
Then there is an integer $n_0\geq 1$ such that $\G(V_{n_0}')\subset \P$.
\end{lemma}

\begin{proof}
Since $\P$ is large, it is the stabilizer of a generalized flag
\[\Ff=\{F_0=0\subset F_1\subset\ldots\subset F_{p_0-1}\subset F_{p_0}\subset\ldots\subset F_p=V\}\]
such that $\dim F_k/F_{k-1}<+\infty$ for all $k\not=p_0$. In particular, $F_{p_0-1}$ is finite dimensional hence
there is $n_1\geq 1$ with $F_{p_0-1}\subset V_{n_1}$. 
By Proposition \ref{P1bis}, the generalized flag $\Ff$ is weakly $E$-compatible.
Since $F_{p_0}$ is finite codimensional, we have $e\in F_{p_0}$ for all but finitely many vectors $e\in E$, hence there is $n_2\geq 1$ with $F_{p_0}\supset V'_{n_2}$. Then the integer $n_0:=\max\{n_1,n_2\}$ is as required in the statement.
\end{proof}

\begin{lemma}
\label{L6bis}
Let $\P\subset \G$ be a semilarge splitting parabolic subgroup which contains the splitting Cartan subgroup $\H(E)$.
For all $m\geq 1$, there is an integer $n_1\geq 1$ such that for every $m$-dimensional subspace $M\subset V$, we can find
an element $g\in\P$ with $g(M)\subset V_{n_1}$.
\end{lemma}

\begin{proof}
Being semilarge, $\P$ is the stabilizer of a generalized flag
\[
\Ff=\{F_0=0\subset F_1\subset\ldots\subset F_p=V\}.
\]
This generalized flag being $E$-compatible,
for all $k\in\{1,\ldots,p\}$ we have a subset $E'_k\subset E$ with $F_k=F_{k-1}\oplus\span{E'_k}$.
The subgroup
\[
\P':=\{g\in\P:g(\span{E'_k})=\span{E'_k}\ \mbox{for all $k\in\{1,\ldots,p\}$}\}=\P\cap\prod_{k=1}^p\GL(\span{E'_k})
\]
acts on the set of vectors
$\bigcup_{k=1}^p\span{E'_k}$
with finitely many orbits, hence there is an integer $\varphi(0)\geq 1$ with
\[
\bigcup_{k=1}^p\span{E'_k}\subset \P'\cdot V_{\varphi(0)}.
\]
For all $n\geq 1$, by applying the same property to the intersection $\P\cap\G(V'_n)$, which is a semilarge splitting parabolic subgroup of $\G(V'_n)$ that contains $\H(E\setminus E_n)$, we get an integer $\varphi(n)>n$ satisfying
\[
\bigcup_{k=1}^p\span{E'_k\cap V'_n}\subset (\P'\cap \G(V'_n))\cdot V_{\varphi(n)}.
\]
By induction, this easily implies that $(\bigcup_{k=1}^p\span{E'_k})^\ell\subset \P'\cdot (V_{\varphi^\ell(0)})^\ell$ for all $\ell\geq 1$. 
Hence
\[V^m=\Big(\bigoplus_{k=1}^p\span{E'_k}\Big)^m\subset \P'\cdot (V_{n_1})^{m}\subset \P\cdot (V_{n_1})^{m}\]
where $n_1:=\varphi^{mp}(0)$.
The lemma ensues.
\end{proof}

By Theorem \ref{T1}, $\X_2\times \X_3$ has a finite number of $\G$-orbits. For showing that $\X=\X_1\times\X_2\times\X_3$ has a finite number of $\G$-orbits,
it suffices to show that $\G$ has a finite number of orbits on $\X_1\times\mathbb{O}$ for every $\G$-orbit $\mathbb{O}\subset\X_2\times\X_3$.
To do this, we fix an element $(\Ff_2,\Ff_3)\in\mathbb{O}\subset\X_2\times\X_3$ and consider its stabilizer
\[\mathbf{S}:=\{g\in\G:g(\Ff_2)=\Ff_2\mbox{ and }g(\Ff_3)=\Ff_3\}=\Stab_\G(\Ff_2)\cap\Stab_\G(\Ff_3).\]
It suffices to show that $\mathbf{S}$ has a finite number of orbits on $\X_1$.

Fix $n_0\geq 1$, large enough so that $\Ff_2$ and $\Ff_3$ belong to $\X_2(n_0)$ and $\X_3(n_0)$, respectively, and such that
\[\G(V'_{n_0})\subset\Stab_\G(\Ff_2)\]
(see Lemma \ref{L6}). Then
\[\mathbf{S}':=\mathbf{S}\cap\G(V'_{n_0})=\Stab_\G(\Ff_3)\cap \G(V'_{n_0}).\]
Since $\Ff_3$ belongs to $\X_3(n_0)$, the chain of subspaces
\[\Ff_3(V'_{n_0}):=\{F\cap V'_{n_0}:F\in\Ff_3\}\]
is an $(E\setminus E_{n_0})$-compatible generalized flag in the space $V'_{n_0}$.
Note that in type (BCD) this generalized flag is $\omega'_{n_0}$-isotropic.
Moreover, the fact that $\Ff_3$ belongs to $\X_3(n_0)$ guarantees that each subspace $F\in\Ff_3$ satisfies $F=(F\cap V_{n_0})\oplus(F\cap V'_{n_0})$. 
We deduce that
\begin{eqnarray*}
\mathbf{S}' & = & \Stab_\G(\Ff_3)\cap\G(V'_{n_0}) \\ & = & \{g\in \G(V'_{n_0}):g(F)=F\ \mbox{for all $F\in\Ff_3$}\} \\
 & = & \{g\in \G(V'_{n_0}):g(F)=F\ \mbox{for all $F\in\Ff_3(V'_{n_0})$}\} \\
 & = & \Stab_{\G(V'_{n_0})}(\Ff_3(V'_{n_0})).
\end{eqnarray*}
Consequently, $\mathbf{S}'$ is a semilarge splitting parabolic subgroup of $\G(V'_{n_0})$ that contains the splitting Cartan subgroup $\H(E\setminus E_{n_0})$.
This key observation is used in the proof of Claims 2 and 3 below.

By (\ref{8}) we know that for every $n\geq n_0$, the (finite-dimensional) subvariety $\X_1(n)$ intersects only finitely many $\mathbf{S}$-orbits.
For completing the proof of Theorem \ref{T2}, it suffices to prove the following claim.

\medskip
{\bf Claim 1:} There is $n_1\geq n_0$ such that, for all $\Ff\in\X_1$, there is $g\in \mathbf{S}$ with  $g\Ff\in\X_1(n_1)$.

\medskip
The combination of Claims 2 and 3 below yields Claim 1, and will make the proof of the theorem complete.

Since $\P_1$ is a large parabolic subgroup, every point $\Ff\in\X_1$ is a finite flag of the form
\[\Ff=\{F_0=0\subset F_1\subset \ldots\subset F_{p_0-1}\subset F_{p_0}\subset\ldots\subset F_p=V\}\]
with $\dim F_k=d_k<+\infty$ for $0\leq k\leq p_0-1$ and $\dim V/F_k=d'_k<+\infty$ for $p_0\leq k\leq p$. 

\medskip
{\bf Claim 2:} There exists $n_1\geq n_0$ such that, for all $\Ff=\{F_0,\ldots,F_p\}\in\X_1$, there is $g\in\mathbf{S}\cap\G(V'_{n_0})$ with $g(F_{p_0-1})\subset V_{n_1}$.

\medskip
{\bf Claim 3:} There exists $n_2\geq n_1$ such that, for all $\Ff=\{F_0,\ldots,F_p\}\in\X_1$ satisfying the condition $F_{p_0-1}\subset V_{n_1}$, there is $g\in\mathbf{S}\cap\G(V'_{n_1})$ with $g(F_k)\supset V'_{n_2}$ for $p_0\leq k\leq p$. This implies that $g\Ff\in\X_1(n_2)$.

\begin{proof}[Proof of Claim 2]
If $\Ff\in\X_1$, then $\pi'_{n_0}(F_{p_0-1})$ is a subspace of $V'_{n_0}$ whose dimension is at most $d_{p_0-1}$.
We have noted that $\mathbf{S}'=\mathbf{S}\cap\G(V'_{n_0})$ is a semilarge splitting parabolic subgroup of $\G(V'_{n_0})$ which
contains $\H(E\setminus E_{n_0})$. 
Hence, Lemma \ref{L6bis} yields an integer $n_1\geq n_0$
such that for every $\Ff\in\X_1$ there is an element $g\in\mathbf{S}'$ with $g(\pi'_{n_0}(F_{p_0-1}))\subset V_{n_1}$.
Since $g\circ \pi'_{n_0}=\pi'_{n_0}\circ g$, we obtain that $\pi'_{n_0}(g(F_{p_0-1}))\subset V_{n_1}$, and hence that $g(F_{p_0-1})\subset V_{n_1}$. This establishes Claim~2.
\end{proof}

\begin{proof}[Proof of Claim 3]
Let $\tilde\X_1$ be the subset of generalized flags $\Ff=\{F_k\}_{k=0}^p\in\X_1$ such that $F_{p_0-1}\subset V_{n_1}$. 
For $\Ff=\{F_k\}_{k=0}^p\in\tilde\X_1$ we define
\[
\Ff(V'_{n_1}):=\{F_0\cap V'_{n_1}\subset F_1\cap V'_{n_1}\subset\ldots\subset F_p\cap V'_{n_1}\}.
\]
This is a weakly $(E\setminus E_{n_1})$-compatible generalized flag in the space $V'_{n_1}$, that is $\omega'_{n_1}$-isotropic in type (BCD). Moreover, we have
\[\dim F_k\cap V'_{n_1}\leq \dim F_{p_0-1}=d_{p_0-1}<+\infty\]
for $0\leq k\leq p_0-1$ and
\[\dim V'_{n_1}/F_k\cap V'_{n_1}\leq \dim V/F_{p_0}=d'_{p_0}<+\infty\]
for $p_0\leq k\leq p$. These observations show that the image of the map $\Ff\in\tilde{\X}_1\mapsto\Ff(V'_{n_1})$ is contained in a union $\tilde\X'_1\cup\ldots\cup\tilde\X'_r$ where each $\tilde\X'_j$ is an ind-variety of generalized flags in the space $V'_{n_1}$ corresponding to a large splitting parabolic subgroup of $\G(V'_{n_1})$.

Since $n_1\geq n_0$, arguing in the same way as for $\mathbf{S}'$, we see that the subgroup
\[\tilde{\mathbf{S}}':=\mathbf{S}\cap\G(V'_{n_1})\]
is a semilarge splitting parabolic subgroup of $\G(V'_{n_1})$. By Theorem \ref{T1}, $\tilde{\mathbf{S}}'$ has a finite number of orbits on every $\tilde\X'_j$. It follows that the set
\[\{\Ff(V'_{n_1}):\Ff\in\tilde\X_1\}\]
intersects finitely many $\tilde{\mathbf{S}}'$-orbits. Hence we can find $n_2\geq n_1$ such that for every $\Ff=\{F_0,\ldots,F_p\}\in\tilde\X_1$, there is $g\in\tilde{\mathbf{S}}'$ with $g(F_{p_0}\cap V'_{n_1})\supset V'_{n_2}$. Whence $g(F_{p_0})\supset V'_{n_2}$. 

The conclusion that $g\Ff=\{g(F_k)\}_{k=0}^p$ belongs to $\X_1(n_2)$ is obtained by observing that
\[
g(F_k)
=
\left\{
\begin{array}{ll}
F_k\subset V_{n_2} & \mbox{if $0\leq k\leq p_0-1$,} \\
(g(F_k)\cap V_{n_2})\oplus V'_{n_2} & \mbox{if $p_0\leq k\leq p$.}
\end{array}
\right.
\]
The proof of Claim 3 is complete.
\end{proof}


\section{Proof of Theorem \ref{T3}}

\label{section-8}

Let $\ell\geq 4$. 
Theorem \ref{T1} implies that $\X$ has infinitely many $\G$-orbits whenever at least two of the splitting parabolic subgroups $\P_1,\ldots,\P_\ell$ are not large.
Hence we may assume that $\P_1,\ldots,\P_{\ell-1}$ are large and consider the construction of Section \ref{exhaustion}.

Since $\ell\geq 4$, it follows from the results in \cite{MWZ1,MWZ2,Matsuki1,Matsuki2} that every (finite-dimensional) multiple flag variety $\X(n)$ has infinitely many $\G(n)$-orbits whenever $n\geq 1$ is large enough. By Lemma \ref{L-key}
we infer that $\X$ has infinitely many $\G$-orbits, completing the proof of Theorem \ref{T3}.

\begin{appendix}

\section{On the classification given in Table \ref{table1}}
\label{appendix}

In this appendix we go over the details of the classification of triples of (proper) splitting parabolic subgroups $(\P_1,\P_2,\P_3)$ of $\G$ such that the ind-variety
\[\X=\G/\P_1\times\G/\P_2\times\G/\P_3\]
has a finite number of $\G$-orbits,
stated in the last part of Theorem \ref{T2} and explicitly listed in Table \ref{table1}. 
The first part of Theorem \ref{T2} reduces this classification to the finite-dimensional case.
This enables us to use the classification results for triple flag varieties of finite type in \cite{MWZ1,MWZ2,Matsuki1,Matsuki2}.
Note that these latter results are not used elsewhere in our paper.

According to the first claim in Theorem \ref{T2}, we assume that for each $i\in\{1,2,3\}$, $\P_i$ is a semilarge parabolic subgroup obtained as the stabilizer of a generalized flag
\[\Ff_i=\{F_{i,0}=0\subsetneq F_{i,1}\subsetneq\ldots\subsetneq F_{i,m_i-1}\subsetneq F_{i,m_i}=V\}\]
(such that $F_{i,k}^\perp=F_{i,m_i-k}$ in type (BCD)).
In the notation of Table \ref{table1} we have
$\length{\P_i}=m_i$, while $\Lambda(\P_i)$ is the list of the dimensions $\dim F_{i,k}/F_{i,k-1}$ (for $k\in\{1,\ldots,m_i\}$) written in nonincreasing order. Some of these dimensions may be infinite, in which case the sequence $\Lambda(\P_i)$ takes the form
\[
\Lambda(\P_i)=(\infty^{\ell_i},c_{i,\ell_i+1},\ldots,c_{i,m_i})
\]
with $\ell_i\in\{1,\ldots,m_i\}$ and a nonincreasing sequence of integers $c_{i,\ell_i+1}\geq\ldots\geq c_{i,m_i}$. Moreover, by the first claim in Theorem \ref{T2}, we assume that at least two of the parabolic subgroups are large, i.e., $\ell_i=1$
for at least two $i\in\{1,2,3\}$.

In the setting of Theorem \ref{T2}, the ind-variety $\X$ is exhausted by finite-dimensional triple flag varieties $\X(n):=\X(E_n)=\prod_{i=1}^3\G(E_n)/\P_i(E_n)$,
where $\G(n):=\G(E_n)$ is a finite-dimensional classical algebraic group of the same type as $\G$ and $\P_i(n):=\P_i(E_n)$ is the stabilizer of a flag
\[\Ff_i(n)=\{F_{i,0}(n)=0\subset F_{i,1}(n)\subset\ldots\subset F_{i,m_i-1}(n)\subset F_{i,m_i}(n)=V_n\}\]
of the finite-dimensional space $V_n=\span{E_n}$.
When $n$ is large enough, 
the list of dimensions $\dim F_{i,k}(n)/F_{i,k-1}(n)$ arranged in nonincreasing order is of the form
\[
\Lambda(\P_i(n))=(d_{i,1}(n),\ldots,d_{i,\ell_i}(n),c_{i,\ell_i+1},\ldots,c_{i,m_i})
\]
where, for all $k\in\{1,\ldots,\ell_i\}$, $\{d_{i,k}(n)\}_n$ is a sequence tending to infinity.

The first part of Theorem \ref{T2} asserts that $\X$ has a finite number of $\G$-orbits if and only if $\X(n)$ has a finite number of $\G(n)$-orbits for all $n$
(in fact, by Lemma \ref{L-key} we may assume that $n$ is large). The latter condition can be characterized in terms of the sequences $\Lambda(\P_i(n))$ by using results from \cite{MWZ1,MWZ2,Matsuki1,Matsuki2}. We now complete the verification of Table~\ref{table1} case by case.

\subsection{$\G=\GL(\infty)$} 
In type (A), it is shown in \cite[Theorem 2.2]{MWZ1} that $\X(n)$ has a finite number of $\G(n)$-orbits if and only if one of the following conditions hold up to permutation within the triple $(\P_1,\P_2,\P_3)$:
\begin{itemize}
\item $m_1=m_2=2$ (referred to as type $(D_{r+2})$ in \cite[Theorem 2.2]{MWZ1});
\item $m_1=2$, $m_2=3$, $m_3\in\{3,4,5\}$ (types $(E_6)$, $(E_7)$, $(E_8)$);
\item $m_1=2$, $m_2=3$, and $c_{1,2}=2$ (type $(E_{r+3}^{(a)})$);
\item $m_1=2$, $m_2=3$, and $c_{2,3}=1$ (type $(E_{r+3}^{(b)})$);
\item $m_1=2$ and $c_{1,2}=1$ (type $(S_{q,r})$).
\end{itemize}
Note that type $(A_{q,r})$ of \cite[Theorem 2.2]{MWZ1} does not occur in our setting since we assume that $\P_1,\P_2,\P_3$ are proper subgroups of $\G$, and hence $m_1,m_2,m_3\geq 2$. The above list of conditions yields the part of Table \ref{table1} concerning $\GL(\infty)$.

\subsection{$\G=\Sp(\infty)$} 

In this case the requirement $F_{i,k}^\perp=F_{i,m_i-k}$ for all $k\in\{1,\ldots,m_i\}$ implies that, if $m_i$ is even then the generalized flag $\Ff_i$ contains a Lagrangian subspace $F_{i,\frac{m_i}{2}}$ which is both infinite dimensional and infinite codimensional. Hence,
\begin{equation}
\mbox{if $m_i$ is even, then $\P_i$ is not large.}
\end{equation}
Since we assume that at least two of the parabolic subgroups $\P_1,\P_2,\P_3$ are large, at least two of the numbers $m_1,m_2,m_3$ must be odd. Taking this observation into account, it follows from \cite[Theorem 1.2]{MWZ2} that $\X(n)$ has a finite number of $\G(n)$-orbits if and only if one of the following cases occurs (up to permutation within $(\P_1,\P_2,\P_3)$):
\begin{itemize}
\item $m_1=2$, $m_2=3$, $m_3\in\{3,5\}$ (types $(\mathrm{Sp}E_6)$ and $(\mathrm{Sp}E_8)$ in \cite[Theorem 1.2]{MWZ2});
\item $m_1=2$, $m_2=3$, and $c_{2,2}=c_{2,3}=1$ (type $(\mathrm{Sp}E_{r+3}^{(b)})$);
\item $m_1=m_2=3$, and $c_{1,2}=c_{1,3}=1$ (type $(\mathrm{Sp}Y_{r+4})$).
\end{itemize}
This corresponds to the part of Table \ref{table1} concerning $\Sp(\infty)$.

\subsection{$\G=\O(\infty)$} Similarly to the case of $\Sp(\infty)$, the parabolic subgroup $\P_i$ is large only if $m_i$ is odd, and
the fact that at least two of the subgroups $\P_1,\P_2,\P_3$ are large implies that at least two of the numbers $m_1,m_2,m_3$ are odd.
Note that, if some $m_i$ is even then the finite-dimensional space $V_n=\span{E_n}$ has even dimension since it contains a Lagrangian subspace. 

If $\dim V_n$ is odd, say $\dim V_n=2m+1$, then \cite[Theorem 1.6]{Matsuki1} implies that $\X(n)$ has a finite number of $\G(n)$-orbits if and only if one of the following situations occurs (up to permutation within $(\P_1,\P_2,\P_3)$): 
\begin{itemize}
\item $m_1=m_2=3$ and $c_{1,2}=c_{1,3}=1$ (this corresponds to (II) in \cite[Theorem 1.6]{Matsuki1});
\item $m_1=m_2=3$, $m_3\in\{3,5\}$, and $\Lambda(\P_2(n))=(m,m,1)$, which means that $\Lambda(\P_2)=(\infty,\infty,1)$  (cases (III) and (IV) of \cite[Theorem 1.6]{Matsuki1}).
\end{itemize}
In the second situation $\P_2$ is not large, hence $\P_1,\P_3$ have to be large, in particular  $\Lambda(\P_1)=(\infty,b,b)$ for some positive integer $b$. Note that condition (I) of \cite[Theorem 1.6]{Matsuki1} cannot be fulfilled here because $\P_1$ or $\P_2$ must be large.

If $\dim V_n$ is even, say $\dim V_n=2m$, then we get from \cite[Theorem 1.7]{Matsuki2} that  $\G(n)$ has a finite number of orbits on $\X(n)$ precisely in the following cases (up to permutation within $(\P_1,\P_2,\P_3)$):
\begin{itemize}
\item $m_1=3$, $m_2\in\{2,3,4\}$, and $c_{1,2}=c_{1,3}=1$ (cases (I-1) and (I-2) of \cite[Theorem 1.7]{Matsuki2});
\item $m_1=2$ and $\Lambda(\P_2(n))$ is one of the sequences $(2m-2,1^2)$, $(2m-4,2^2)$, $(2m-6,3^2)$, $(2m-4,1^4)$, which equivalently means that $\Lambda(\P_2)$ is $(\infty,b^2)$ with $b\in\{1,2,3\}$ or $(\infty,1^4)$ (case (II) of \cite[Theorem 1.7]{Matsuki2});
\item $m_1=2$, $m_2=3$ with $\Lambda(\P_2(n))=(2m-2b,b^2)$ for an integer $b\geq 4$, which means that $\Lambda(\P_2)=(\infty,b^2)$, and $m_3\in\{3,5\}$ (cases (III-1) and (III-2) of \cite[Theorem 1.7]{Matsuki2});
\item $m_1=2$, $m_2=3$ with $\Lambda(\P_2(n))=(2m-2b,b^2)$ for an integer $b\geq 4$, which means that $\Lambda(\P_2)=(\infty,b^2)$, and $\Lambda(\P_3(n))$ is $(2m-2c-4,c^2,1^4)$ with $c\geq 1$ or ($2m-8,1^8)$, which means that $\Lambda(\P_3)$ is $(\infty,c^2,1^4)$ or $(\infty,1^8)$ (cases (III-3) and (III-4) of \cite[Theorem 1.7]{Matsuki2}).
\end{itemize}
Note that \cite[Theorem 1.7]{Matsuki2} contains more cases, but we disregard the cases which prevent two of the subgroups $\P_1,\P_2,\P_3$ to be large.

Altogether, the conditions listed in this subsection yield the part of Table \ref{table1} concerning $\O(\infty)$.
\end{appendix}

\end{document}